\documentclass[12pt]{amsart}
\usepackage{geometry}                
\usepackage{graphicx,mathrsfs,epstopdf}
\usepackage{amsmath,amsfonts,amssymb,amsthm}
\usepackage[colorlinks,citecolor=blue,urlcolor=blue]{hyperref}

\newtheorem{theorem}{Theorem}[section]
\newtheorem{lem}[theorem]{Lemma}
\newtheorem{prop}[theorem]{Proposition}
\newtheorem{cor}[theorem]{Corollary}

\theoremstyle{definition}

\newtheorem{example}[theorem]{Example}

\theoremstyle{remark}
\newtheorem{rem}[theorem]{Remark}

\newcommand{\QED}{\ifhmode\unskip\nobreak\fi\quad {\rm Q.E.D.}} 

\newcommand{\cL}{\mathcal{L}}

\newcommand{\RR}{\mathbb{R}}

\newcommand{\de}{{\rm de}}

\title{Brownian motion tree models are toric}
\author[]{Bernd Sturmfels}
\address{Max Planck-Institute for Mathematics in the Sciences, Leipzig, Germany}
\email{bernd@mis.mpg.de}

\author[]{Caroline Uhler}
\address{Laboratory for Information and Decision Systems, and Institute for Data, 
Systems, and Society, Massachusetts Institute of Technology, Cambridge, MA, USA}
\email{cuhler@mit.edu}

\author[]{Piotr Zwiernik}
\address{Dept.~of Economics and Business, Universitat Pompeu Fabra, Barcelona, Spain}
\email{piotr.zwiernik@upf.edu}

\thanks{BS acknowledges funding from the Einstein Foundation Berlin
and the NSF (DMS-1419018).
CU was partially supported by NSF (DMS-1651995), ONR (N00014-17-1-2147 and N00014-18-1-2765), IBM, and a Sloan Fellowship. 
PZ was supported by the Spanish Ministry of Economy and Competitiveness (MTM2015-67304-P), a Beatriu de Pin\'{o}s fellowship (2016 BP 00002), and 
the Ayudas Fundaci\'on BBVA.
This project started  in December 2018 in Bristol,
at the workshop {\em Advances in Applied Algebraic Geometry}.
Many thanks to Fatemeh Mohammadi for organizing it}

\begin{document}

\begin{abstract}
Felsenstein's classical model for Gaussian distributions on
a phylogenetic tree is 
shown to be a toric variety in the space of concentration matrices.
We present an exact semialgebraic characterization  of this model,
 and we demonstrate
how the toric structure leads to exact methods for maximum likelihood estimation. Our results also give new insights into the geometry of ultrametric~matrices.
\end{abstract}

\maketitle

\section{Introduction}

Brownian motion tree models are classical statistical models for phylogenetic trees.
They were  introduced by Felsenstein \cite{felsenstein_maximum-likelihood_1973}
to examine continuous measurements of phenotypes in evolutionary biology.
The vertices of the tree represent real-valued random variables,
whose joint distribution obeys a Gaussian law.

Let $\tilde T$ be a tree with $n+1$ leaves, labelled $0,1,\ldots,n$,
and with no vertices of degree two. Let
$T$ be the rooted tree obtained from $\tilde T$ by directing all edges away from $0$.
The set $V$ of non-root vertices of $T$ is in natural bijection with the set of edges of~$T$.
  A vertex $u\in V$ is a \emph{descendant} of $v\in V$ if there is a directed path from $v$ to $u$. 
   The set of all leaves of $T$ that are   descendants of $v$ is denoted by $\de(v)$.
   We fix a total order on $V$ such that  $u\leq v$ if $\de(u)\subseteq  \de(v)$.  Given $u,v\in V$, we write $w={\rm lca}(u,v)$ for
   their  most recent common ancestor.
Figure~\ref{fig:running0} shows our running example.

\begin{figure}[htp!] \quad
\includegraphics[height=4.7cm]{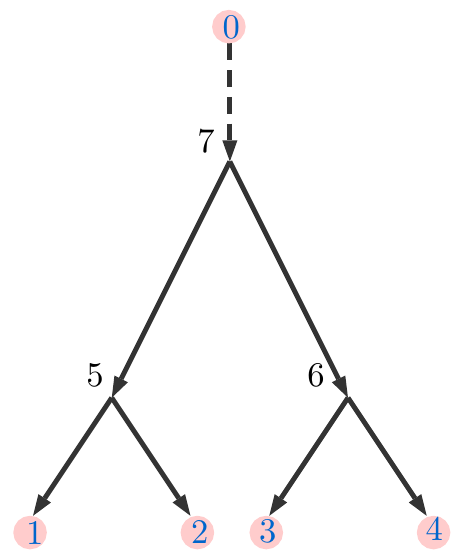} \qquad 
\includegraphics[height=3cm]{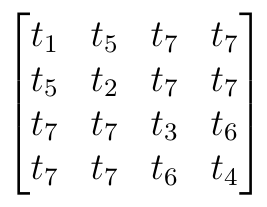} \qquad
\includegraphics[height=4.7cm]{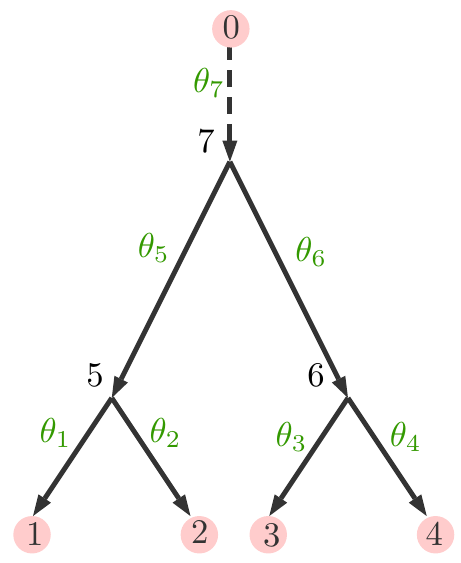}
	 \vspace{-0.13in}
	\caption{A tree $T$ with $n{=}4$ leaves, $|V|{=}7$ edges,
	and its matrix representation.}\label{fig:running0}
\end{figure}	

In the space $\mathbb S^n$ of symmetric $n\times n$ matrices
$\Sigma=(\sigma_{ij})$ we consider the subspace
$$
\mathcal L_T\;=\;\bigl\{\,\Sigma\in \mathbb S^n:\; \sigma_{ij}=\sigma_{kl} \,
\mbox{ if } \, {\rm lca}(i,j) = {\rm lca}(k,l) \,\bigr\}.
$$
Using parameters $(t_v: v\in V)$, 
the matrices in $\mathcal{L}_T$ satisfy
$\sigma_{ij} = t_v$ for $v = {\rm lca}(i,j)$.
This furnishes a representation of the tree $T$ by a matrix, as shown  in
Figure~\ref{fig:running0}.

We are interested in Gaussian distributions on $\RR^n$
with covariance matrix in $\mathcal{L}_T$.
Their concentration matrices $ K =(\kappa_{ij})$ form the
$|V|$-dimensional algebraic variety
$$ \mathcal{L}_T^{-1} \,\, = \,\, \bigl\{ \,K=\Sigma^{-1} \,: \, \Sigma \in \mathcal{L}_T \,\bigr\}
\quad \subset \,\,\, \mathbb{S}^n.$$
We identify $\mathcal{L}_T^{-1}$ with its Zariski closure in the 
projective space $\mathbb{P}(\mathbb{S}^n) \simeq
\mathbb{P}^{\binom{n+1}{2}-1}$.
In this paper we show that  the variety $\mathcal{L}_T^{-1}$ is linearly isomorphic to
a  toric variety in $ \mathbb{P}^{\binom{n+1}{2}-1}$.
In tropical geometry \cite[Remark 4.3.11]{MS}
 and algebraic combinatorics \cite[Theorem 4.6]{lara},
 one associates
 a toric ideal $I_{\tilde T}$ with the unrooted tree
 ${\tilde T}$ as follows.
 The ideal $I_{\tilde T}$ has  the quadratic generators
 $p_{ik} p_{jl} - p_{il} p_{jk}$ where 
 $\{i,j\}$ and $\{k,l\}$ are cherries in the induced
  $4$-leaf subtree on any quadruple $i,j,k,l \in \{0,1,\ldots,n\}$.
    
To reveal the toric structure,  we introduce a 
change of coordinates in~$\mathbb{S}^n$ as follows:
 \begin{equation}
 \label{eq:pcoord} \begin{matrix} 
 p_{ij} &=& - \kappa_{ij} \quad &\quad {\rm for} \,\,1 \leq i < j \leq n, \\
 p_{0i} & = & \sum_{j=1}^n \kappa_{ij} &  {\rm for} \,\,1 \leq i \leq n .
 \end{matrix}
 \end{equation}
With this, the concentration matrix $K = (\kappa_{ij})$
is the {\em reduced Laplacian} of the complete graph on $n{+}1$ vertices
with edge labels $p_{ij}$.  See \cite[Example 4.9]{MSUZ},
where the matrix for $n=3$ is shown in equation (4.6).
Here is the same scenario for $n=4$:

\begin{example} \label{ex:4k4} We fix coordinates
$\,p_{01},p_{02},\ldots,p_{34}\,$ on $\,\mathbb{P}(\mathbb{S}^4) = \mathbb{P}^9\,$
by setting
$$ \begin{small} K \,\, = \,\,\begin{bmatrix} 
          p_{01} {+} p_{12} {+} p_{13} {+} p_{14}  \! &  -p_{12}  &  -p_{13}  &  -p_{14} \\
          -p_{12}  &  \! p_{02} {+} p_{12} {+} p_{23} {+} p_{24} \! &  -p_{23}  &  -p_{24} \\
          -p_{13}  &  -p_{23}  & \!  p_{03} {+} p_{13} {+} p_{23} {+} p_{34}\!  &  -p_{34} \\
          -p_{14}  &  -p_{24}  &  -p_{34}  & \!  p_{04} {+} p_{14} {+} p_{24} {+} p_{34} 
          \end{bmatrix}\! .
          \end{small}
          $$
Fix the tree $T$ in Figure~\ref{fig:running0}.
The $6$-dimensional toric variety $\mathcal{L}_T^{-1}$ in 
$\mathbb{P}^9$ is defined~by
$$ I_{\tilde T} \, =\,
\langle \,
p_{01} p_{23} - p_{02} p_{13} ,\,
p_{01} p_{24} - p_{02} p_{14}  ,\,
p_{03} p_{14} - p_{04} p_{13} ,\,
p_{03} p_{24} - p_{04} p_{23} ,\,
p_{13} p_{24} - p_{14} p_{23} \,
\rangle.
$$
These quadrics vanish for the inverse of any matrix with the structure in
Figure~\ref{fig:running0}.~\qed
\end{example}

The title of this paper is an abridged version of the following statement:

\begin{theorem} \label{thm:main}
The variety $\mathcal{L}^{-1}_T$ of concentration matrices in the
Brownian motion tree model,  in coordinates (\ref{eq:pcoord}),
 coincides with the toric variety defined by the ideal~$ I_{\tilde T}$.
\end{theorem}

The proof of this theorem will be given in Section~\ref{sec:toric}.
First, however, in Section~\ref{sec:BMT}, we offer an
introduction to the statistical model and its phylogenetic applications.
Our statistical models correspond to semialgebraic subsets of 
$\mathcal{L}_T$ or $\mathcal{L}^{-1}_T$. We are interested in two
subsets of $\mathcal{L}_T$, namely 
 the {\em spectrahedron} $\mathcal L_T \,\cap \,\mathbb S_+^n$, obtained by intersection with
 the cone $\,\mathbb S_+^n\,$ of positive definite matrices, and the  polyhedral cone 
 $$
\mathcal L_{T,\geq}\;=\;\
\bigl\{\,\Sigma\in \mathcal L_T:\;
0 \leq \sigma_{ij}\leq \sigma_{kl}\,\mbox{ whenever} \,{\rm lca}(i,j) \leq{\rm lca}(k,l)
\bigr\}.
$$ 
We shall see that $\mathcal L_{T,\geq}$ is a simplicial cone, contained in
the spectrahedron $\mathcal L_T\cap \mathbb S_+^n$.

Matrices in  $\mathcal L_{T,\geq}$ play an important role in statistics. By Proposition~3.14 in \cite{dellacherie2014inverse}, every matrix $\Sigma$ in  $\mathcal L_{T,\geq}$ is an 
{\em ultrametric matrix} in $\mathbb{S}^n$, i.e.~it satisfies
$\,\sigma_{ij}  \geq \min\{\sigma_{ik},\sigma_{jk}\} \geq 0 \,$
for all $i,j,k$.
By Theorem~3.16, every ultrametric matrix lies in $\mathcal L_{T,\geq}$ for some tree $T$. Ultrametric matrices appear in the potential theory of finite state Markov chains, which 
is the context of \cite{dellacherie2014inverse}. 
Our motivation came from 
phylogenetics \cite{felsenstein_maximum-likelihood_1973}
and Gaussian maximum likelihood estimation~\cite{ZUR14}.

Every matrix $\Sigma$ in  $\mathbb S_+^n$ represents a Gaussian distribution on $\RR^n$.
Both $\mathcal L_T\cap \mathbb S_+^n$  and $\mathcal L_{T,\geq}$ belong to the class
 of linear Gaussian covariance models \cite{andersonLinearCovariance,ZUR14}.

The main result of this paper is  Theorem \ref{thm:mainsemi}.
This is an extension
of Theorem \ref{thm:main} which features toric inequalities
$p_{ik} p_{jl} \leq p_{ij} p_{kl}$ in addition to
the quadratic binomial equations in $I_{\tilde T}$.
It offers an exact semialgebraic description of 
the model $\mathcal L_{T,\geq}^{-1}$ in terms
of the nonnegative coordinates $p_{ij}$.
The proof of this result is presented in Section \ref{sec:semialg}.
It rests on formulas 
that express
$p_{ij}$ in terms of {\em treks} as in~\cite{DST0}.

Section~\ref{sec:mle} is about fitting Brownian motion tree models to data, given by
a sample covariance matrix $S$ in $\mathbb{S}^n_+$.
We do so by maximizing the log-likelihood function 
\begin{equation}
\label{eq:loglike1}
\ell(\Sigma)\;=\;-\log\det \Sigma-{\rm trace}(S \Sigma^{-1}).
\end{equation}
This function is non-convex. The expression in terms of $K=\Sigma^{-1}$ equals
\begin{equation}
\label{eq:loglike2}
\ell(K)\;=\;\log\det K-{\rm trace}(SK).
\end{equation}
This function is convex in $K$, which motivates analyzing maximum likelihood estimation for Brownian motion tree models as an optimization problem over 
$\mathcal L_T^{-1}$. As we will show in Section~\ref{sec:mle}, in this parameterization maximum likelihood estimation boils down to 
solving a system of 
polynomial equations on $\mathcal L_T^{-1}$. 
The paper concludes with a  brief discussion
on how Theorem~\ref{thm:mainsemi} might be applied
to likelihood inference.

\section{Tree models and their parameters}
\label{sec:BMT}

Brownian motion is a stochastic process that characterizes the random motion of particles.
It is a Wiener process $W_t$  satisfying $W_0=0$, with independent increments, and such that $W_t-W_s$ for $t\geq s$ has a Gaussian distribution with mean zero and variance $t-s$.  Brownian motion on a rooted binary tree $T$ can also be described using the Wiener process. The process starts at vertex $0$. At time $t=t_{2n-1}$, it splits into two, and each of the two processes starts evolving independently at value $W_{t_{2n-1}}$.  It again proceeds
 according to the Wiener process until another splitting event occurs. We  think about this process as evolving along $T$, where~the parameters $t_v$ for inner vertices $v$ represent the times of splitting events. This
 construction is a continuous interpretation of the Gaussian \emph{structural equation model}
 (\ref{eq:BMT}) discussed next.

Given a rooted tree $T$, we define a Gaussian distribution on $T$ as follows. First, set $Y_{0}\equiv 0$. Then to each vertex $v\in V$ we associate independently a Gaussian random variable $\epsilon_{v}$ with mean zero and variance $\theta_{v} \geq 0$.
  The corresponding Markov~process on $T$ is a collection of 
  real-valued random variables $Y_{v}$ for $v\in V$.
  They  satisfy 
\begin{equation}\label{eq:BMT}
 Y_{v}=Y_{u}+\epsilon_{v} \quad\mbox{for every edge }\,u\to v\in E.
\end{equation}
Since a linear transformation of a Gaussian vector is also Gaussian,
 we conclude that the random vector $Y=(Y_{v})_{v\in V}$ is Gaussian. The set of covariance matrices of the marginal distributions on the leaf-variables $(Y_1,\ldots,Y_n)$ 
is  the polyhedral cone  $\mathcal L_{T,\geq}$. 
 
\begin{prop}\label{prop:BMTstr}
The random vector $(Y_1,\ldots,Y_n)$ is normally distributed with mean zero, 
and the entries $\sigma_{ij} = {\rm cov}(Y_{i},Y_{j})$
 of its covariance matrix $\Sigma_\theta$ are
\begin{equation}\label{eq:sigmaYuv}
\sigma_{ij} \,\,\;=\;\sum_{v\leq {\rm lca}(i,j)}\theta_{v}
\qquad {\rm for} \,\, i,j=1,\ldots,n.
\end{equation}
The resulting Gaussians on $\RR^n$ are precisely those with
covariance matrices in $\mathcal L_{T,\geq}$.
\end{prop}

\begin{proof}
Using (\ref{eq:BMT}) recursively, we can write each $Y_i$ in terms of the error terms as
$$Y_i \,\,= \,\,\sum_{v\leq i} \epsilon_v.	$$
Equation (\ref{eq:sigmaYuv}) follows from this and the fact that all $\epsilon$'s are mutually independent.  The linear inequalities  $\,\sigma_{ij}\leq \sigma_{kl} \,$
that define the polyhedral cone $\mathcal L_{T,\geq}$ inside
the linear space  $\mathcal L_{T}$
are equivalent to the requirement that the $\theta_i$'s be nonnegative.
 \end{proof}

\begin{example}\label{ex:running1}
Consider the tree in Figure \ref{fig:running0}.
The random variables for the inner vertices of the tree are
$Y_0=0$, $Y_7=\epsilon_7$, $Y_{5}=\epsilon_{7}+\epsilon_{5}$, 
$Y_{6}=\epsilon_{7}+\epsilon_{6}$, and we have
$\,
Y_{1} = \epsilon_{7}+\epsilon_{5}+\epsilon_{1},\;
	Y_{2} = \epsilon_{7}+\epsilon_{5}+\epsilon_{2},\;
	Y_{3} = \epsilon_{7}+\epsilon_{6}+\epsilon_{3},\;
	Y_{4} = \epsilon_{7}+\epsilon_{6}+\epsilon_{4}\,$
	for the leaves.
	
The $\epsilon_v$ are independent univariate Gaussians with mean $0$ and
variance $\theta_v$. Hence the marginal distribution of $(Y_1,Y_2,Y_3,Y_4)$ is Gaussian
 with the  covariance matrix 
\begin{equation}\label{eq:SigmaTheta} \begin{small}
	\Sigma_{\theta}\,\,=\,\,\begin{bmatrix}
	\theta_{1}+\theta_{5}+\theta_{7}& \theta_{5}+\theta_{7}& \theta_{7}& \theta_{7}\\
		\theta_{5}+\theta_{7} & \theta_{2}+\theta_{5}+\theta_{7} & \theta_{7}& \theta_{7}\\
			\theta_{7}& \theta_{7}& \theta_{3}+\theta_{6}+\theta_{7}& \theta_{6}+\theta_{7}\\
				\theta_{7}&\theta_{7} & \theta_{6}+\theta_{7}& \theta_{4}+\theta_{6}+\theta_{7}
	\end{bmatrix}. \end{small}
	\end{equation}
This is the matrix in (\ref{eq:sigmaYuv}) 
and in Figure \ref{fig:running0}. The constraint that the
$\theta_i$ are nonnegative translates into the inequalities $t_1,t_2 \geq t_5$ and
$t_3,t_4 \geq t_6$ and $t_5,t_6 \geq t_7 \geq 0$. \hfill\qed
\end{example}

The extreme rays of the polyhedral cone $\mathcal L_{T,\geq}$ are as follows.
Let $g_v\in \{0,1\}^n$ be the vector with
$(g_v)_i=1$ if $i\in {\rm de}(v)$ and $(g_v)_i=0$ otherwise. The corresponding
rank one matrices $G_v=g_v g_v^T$ form a basis for $\mathcal{L}_T$. In fact,
the matrix in (\ref{eq:sigmaYuv}) equals
\begin{equation}\label{eq:SigmaG}
\Sigma\;=\;\sum_{v\in V}\theta_v G_v .	
\end{equation}

\begin{cor}
The cone $\mathcal L_{T,\geq}$ is a simplicial cone, spanned
by the rank one matrices $G_v$ associated with vertices $v \in V$.
It is contained
in the spectrahedral cone $\mathcal L_T\cap \mathbb S_+^n$.
\end{cor}

Note that this inclusion is strict.
For instance, the matrix $\Sigma_\theta$ in (\ref{eq:SigmaTheta}) 
 is positive definite if we set $\,\theta_1 = \theta_2 = \theta_3 = \theta_4 = 5$,
 $\theta_5 = \theta_6= 0$ and $\theta_7 =- 1$.
This means that the linear covariance model is strictly larger than
the Brownian motion tree model.

We next interpret our model in the context  of distance-based phylogenetics.
Using the natural bijection  between non-root vertices and edges,
we label each edge of $T$ with a parameter  $\theta_v$.
This is shown in the tree on the right in Figure \ref{fig:running0}.
We  think of $\theta_v \geq 0$ as the {\em length} of the associated edge. We
 compute the distance between any two leaves of $\tilde T$ by summing the lengths of edges on the unique path joining them. The 
 collection of resulting distances $d_{ij}$ for $i,j=0,1,\ldots,n$ is a \emph{tree metric} on~$\tilde T$. 

The correspondence between ultrametric $n \times n$ matrices and  tree metrics on $n+1$ taxa
is known in phylogenetics as the \textit{Farris transform}. The formulas are
$$
\left\{\begin{array}{ll}
\sigma_{ii}\;\;=\;\; d_{0i} & \mbox{for }1\leq i\leq n\\
\sigma_{ij}\;\;=\;\; \frac{1}{2}(d_{0i}+d_{0j}-d_{ij}) & \mbox{for }1\leq i<j\leq n, 
\end{array}\right.
$$
and these are  equivalent to (\ref{eq:sigmaYuv}).
The inverse of the Farris transform is given by
$$
\left\{\begin{array}{ll}
d_{0i}\;\;=\;\; \sigma_{ii}& \mbox{for }1\leq i\leq n\\
d_{ij}\;\;=\;\; \sigma_{ii}+\sigma_{jj}-2\sigma_{ij}& \mbox{for }1\leq i\leq j\leq n.
\end{array}\right.
$$

\begin{prop}\label{th:Farris}The 
 model $\cL_{T,\geq}$ is identified with the
cone of tree metrics on $\tilde T$ via the Farris transform
$(d_{ij}) \mapsto (\sigma_{ij})$. The parameters $\theta_v$ are the lengths
of the edges.\end{prop}

\begin{proof}
The diagonal entry $\sigma_{ii}$ of the covariance matrix is the sum of
the lengths $\theta_v$ of the incoming edges for all vertices $v$ on the path from the root $0$
to  leaf $i$. Therefore, $d_{0i} = \sigma_{ii}$ is the distance from $0 $ to $i$
in the unrooted tree $\tilde T$. Each off-diagonal entry $\sigma_{ij}$
is the length of the path from the root $0$ to ${\rm lca}(i,j)$.
Hence $\sigma_{ii} - \sigma_{ij}$ is the length of the path from
${\rm lca}(i,j)$ to the leaf $i$. We conclude that
$d_{ij} = (\sigma_{ii} - \sigma_{ij}) + (\sigma_{jj} - \sigma_{ij})$ is the
length of the path from leaf $i$ to leaf $j$ in $\tilde T$.
Since the Farris transform is an invertible linear transformation,
it identifies the two simplicial cones in $\RR^{\binom{n+1}{2}}$.
\end{proof}

We next turn to the space of all tree metrics, which is a key object in phylogenetics.
   A classical result of Buneman \cite{buneman1971} states that a metric $D=(d_{ij})$ on $\{0,\ldots,n\}$ is a tree metric (for some tree) if and only if it satisfies the \emph{four point condition}:
\begin{equation}\label{eq:4point}
d_{ij}+d_{kl}\;\leq\; \max\{d_{ik}+d_{jl},d_{il}+d_{jk}\}
\quad \hbox{for all $\,i,j,k,l\in \{0,1,\ldots,n\}$.}
\end{equation}
If $D$ is a tree metric on $\tilde T$ then the following additional equation holds:
\begin{equation}\label{eq:4pointeq} \quad
d_{ik}+d_{jl}\;=\;d_{il}+d_{jk} \qquad
\hbox{if $\{i,j\},\, \{k,l\}$ are cherries in the quartet on $i,j,k,l$}.
\end{equation}
The constraints (\ref{eq:4point}) and (\ref{eq:4pointeq}) are well-known also in
tropical geometry \cite[\S 4.3]{MS} where 
one identifies the space of tree metrics with the
tropical Grassmannian that parametrizes tropical  lines in
$\RR^{n+1}/\RR {\bf 1}$. This is related to Theorem 
\ref{thm:main} as follows.

\begin{rem}
If we set $p_{ij}=e^{-d_{ij}}$ then
the linear relations (\ref{eq:4pointeq})
that hold for tree metrics on ${\tilde T}$ are precisely
the equations $p_{ik} p_{jk} = p_{il} p_{jk}$ that define the
toric ideal $I_{\tilde T}$.
\end{rem}

We now state our main result. It augments Theorem \ref{thm:main} 
by incorporating the inequalities in (\ref{eq:4point}).
The unrooted tree obtained from $\tilde{T}$ by restricting to any four leaves $i,j,k,l$ 
is called a {\em quartet} of $\tilde{T}$. If equality holds in 
(\ref{eq:4point}) then this four-leaf tree is a {\em star quartet}.
If the inequality in (\ref{eq:4point}) is strict then
we call it a {\em trivalent quartet}.

\begin{theorem} \label{thm:mainsemi}
Given any rooted tree $T$,
the set $\mathcal{L}_{T,\geq}^{-1}$ of concentration matrices in the Brownian motion
tree model is the set of positive definite matrices $K$ satisfying
\begin{equation}\label{eq:constraints}
\begin{matrix} &
p_{ij}\, \geq \,0 & \hbox{for all}  \quad 0 \leq i < j \leq n, \\
&	p_{ik}p_{jl}\;=\;p_{il}p_{jk}\;= \; p_{ij}p_{kl} & \hbox{for all star quartets} \,\,\, ijkl, 
		 \\ \hbox{and} \quad &
	p_{ik}p_{jl}\;=\;p_{il}p_{jk}\;\leq\; p_{ij}p_{kl}		
	 & \quad \hbox{for all trivalent quartets}\,\,\, ij|kl.
	\end{matrix}
	\end{equation}
\end{theorem}

\begin{rem}
These  inequalities are satisfied by 
 $p_{ij}=e^{-d_{ij}}$ where $(d_{ij})$ is any tree metric on $\tilde T$.
Thus, the set of models $\mathcal{L}_{T,\geq}^{-1}$,
where $T$ ranges over all rooted trees on $n$ leaves,
is a multiplicative realization of the space of phylogenetic trees.
This is reminiscent of the
{\em space of phylogenetic oranges} studied by Moulton and Steel~\cite{MoSt}. 
\end{rem}

We illustrate the contents of Theorem~\ref{thm:mainsemi} for our running example.

\begin{example} \label{ex:5_15}
Fix the tree in Figure \ref{fig:running0}
with covariance matrix  $\Sigma_\theta$ in (\ref{eq:SigmaTheta}). 
Set~$s := {\rm det}(\Sigma_\theta)>0$.
Writing the concentration matrix $ K = \Sigma_\theta^{-1}$ as in Example \ref{ex:4k4}, we have
$$ 
\begin{small}
\begin{matrix}
 p_{13} s \,=\, \theta_2 \theta_4 \theta_7, \,\,
 p_{14} s \,=\, \theta_2 \theta_3 \theta_7 , \,\,
 p_{23} s \,=\, \theta_1 \theta_4 \theta_7 , \,\,
 p_{24} s \,=\, \theta_1 \theta_3 \theta_7 , \,\,
(p_{03} p_{12} -p_{02} p_{13}){s} = \theta_4 \theta_5 , \\
(p_{04} p_{12}-p_{02}p_{14}){s} \,=\, \theta_3 \theta_5\,, \,\,\,
(p_{01} p_{34}-p_{04} p_{13}){s} \,=\, \theta_2 \theta_6 \,,\,\,\,
(p_{02} p_{34}-p_{04} p_{23}){s} \,=\, \theta_1 \theta_6 \, ,
\smallskip \\
(p_{12} p_{34}-p_{14} p_{23}){s} \,=\, \theta_5 \theta_6+\theta_5 \theta_7+\theta_6 \theta_7\, ,
\\
p_{01} s \,=\, (\theta_3 \theta_4+\theta_3 \theta_6+\theta_4 \theta_6) \theta_2 \, ,\quad
p_{02} s \,=\, (\theta_3 \theta_4+\theta_3 \theta_6+\theta_4 \theta_6) \theta_1\, , 
\\
p_{03} s \,=\, (\theta_1 \theta_2+\theta_1 \theta_5+\theta_2 \theta_5) \theta_4 \, , \quad
p_{04} s\,=\, (\theta_1 \theta_2+\theta_1 \theta_5+\theta_2 \theta_5) \theta_3 \, , \\
p_{12}s \,=\, \theta_3 \theta_4 \theta_5+\theta_3 \theta_4 \theta_7+\theta_3 \theta_5 \theta_6+\theta_3 \theta_5 \theta_7+\theta_3 \theta_6 \theta_7+\theta_4 \theta_5 \theta_6+\theta_4 \theta_5 \theta_7+\theta_4 \theta_6 \theta_7\, , \\
p_{34}s \,=\, \theta_1 \theta_2 \theta_6+\theta_1 \theta_2 \theta_7+\theta_1 \theta_5 \theta_6+\theta_1 \theta_5 \theta_7+\theta_1 \theta_6 \theta_7+\theta_2 \theta_5 \theta_6+\theta_2 \theta_5 \theta_7+\theta_2 \theta_6 \theta_7.
\end{matrix}
\end{small}
$$
The five quadratic binomials in $I_{\tilde T}$ are zero for these $p_{ij}$.
Assuming this, Theorem~\ref{thm:mainsemi} says that
these $15$ expressions are nonnegative if and only if
$\theta_1,\ldots,\theta_7 \geq 0$. 
\end{example}

\section{Toric ideals from trees} \label{sec:toric}

In this section we prove Theorem~\ref{thm:main}. The proof of  Theorem~\ref{thm:mainsemi} is given in Section~\ref{sec:semialg}. The following 
code in {\tt Macaulay2} \cite{M2} provides the quadratic generators for our running example. It also shows that the rooted tree $T$ need not be binary. 

\begin{example}  \label{ex:M2} Example~\ref{ex:4k4}
can be verified in {\tt Macaulay2} \cite{M2} by running this code:
\begin{verbatim}
R = QQ[t1,t2,t3,t4,t5,t6,t7,p01,p02,p03,p04,p12,p13,p14,p23,p24,p34];
S = matrix {{t1,t5,t7,t7},
            {t5,t2,t7,t7},
            {t7,t7,t3,t6},
            {t7,t7,t6,t4}};
K = matrix {{p01+p12+p13+p14, -p12, -p13, -p14},
            {-p12, p02+p12+p23+p24, -p23, -p24},
            {-p13, -p23, p03+p13+p23+p34, -p34},
            {-p14, -p24, -p34, p04+p14+p24+p34}};
id4 = matrix {{1,0,0,0},{0,1,0,0},{0,0,1,0},{0,0,0,1}};
I = eliminate({t1,t2,t3,t4,t5,t6,t7},minors(1,S*K-id4))
codim I, degree I, betti mingens I
\end{verbatim}
As claimed, the toric ideal has codimension $3$, degree $5$ and five quadratic generators.

We now examine non-binary trees. First we replace the two occurrences of {\tt t6}
by {\tt t7} in the covariance matrix {\tt S}. The resulting tree has $|V| = 6$. By running
the modified {\tt Macaulay2} code, we see that the ideal is still toric.  It has
codimension $4$, degree $8$ and 7 quadratic generators. Finally, we replace
both  {\tt t5} and {\tt t6} with~{\tt t7}. 
Now the unrooted tree $\tilde T$ has $|V| = 5$. It
is the {\em star tree} with leaves $0,1,2,3,4$. Its toric ideal $I_{\tilde T}$
is the ideal of the {\em second hypersimplex}. It has codimension
$5$ and degree $11$, with 10 quadratic generators. 
Modifying the code confirms these data. \hfill\qed
\end{example}
  
  \begin{proof}[Proof of Theorem \ref{thm:main}]
  We use the following parametric representation for the toric variety
  of the ideal $I_{\tilde T}$ associated with the
 unrooted tree ${\tilde T}$. It is given by
 Laurent monomials in the entries $t_v$ of the matrix representation of
the rooted tree~$T$:
\begin{equation}
\label{eq:toricpara}
\begin{matrix} 
 p_{ij} & \mapsto & t_{{\rm lca}(i,j)}/(t_i t_j) &\quad {\rm for} \,\,1 \leq i < j \leq n, \\
 p_{0i} & \mapsto & 1/t_i &  {\rm for} \,\,1 \leq i \leq n .
 \end{matrix}
 \end{equation}
The ideal $I_{\tilde T}$ is the kernel of the ring homomorphism
$\RR[p]\rightarrow \RR[t^{\pm}]$ given by (\ref{eq:toricpara}). 

The variety $\mathcal{L}_T^{-1}$ is a cone in $\mathbb S^n$
given parametrically by mapping a covariance matrix $\Sigma$ to
its inverse $K =\Sigma^{-1}$. Since the parametrization
is homogeneous, we may replace the inverse by the adjoint.
By slight abuse of notation we set
$K = {\rm det}(\Sigma) \cdot \Sigma^{-1}$.
The entries $\kappa_{ij}$ of the matrix $K$ are homogeneous polynomials
of degree $n-1$ in the parameters $t =(t_v)$ for $v\in V$. The same holds
for the coordinates $p_{ij}$ in (\ref{eq:pcoord}). We write
$P_{ij}(t)$ for these homogeneous polynomials.
Our claim states that the toric ideal $I_{\tilde T}$ coincides
with the kernel of the ring homomorphism $\RR[p]\rightarrow \RR[t], \,
p_{ij} \mapsto P_{ij}(t)$.

To prove this, we examine the initial monomials and the irreducible 
factorization of the polynomials $P_{ij}(t)$. Here we fix the 
degree reverse lexicographic order on $\RR[t]$ given by
$t_u\succ t_v$ if $u\leq  v$ in $T$. For $1 \leq i < j \leq n$,
the polynomial $P_{ij}(t)$ is equal (up to sign) to the determinant of the
$(n-1) \times (n-1)$ submatrix of $\Sigma$ that is obtained by
deleting row $i$ and column $j$.
The initial monomial
is the product of the entries of that submatrix which appear along the main diagonal.
To be precise, we find
$$ {\rm in}(P_{ij}(t)) \,\,=\,\, t_1 t_2 \cdots t_n \cdot t_{{\rm lca}(i,j)}/(t_i t_j) .$$

The polynomial $P_{0i}(t)$ is the determinant of the $n \times n$ matrix
obtained from $\Sigma$~by replacing the $i$th row with the
all-ones vector $(1,1,\ldots,1)$. Its initial monomial~equals
$$ {\rm in}(P_{0i}(t)) \,\,=\,\, t_1 t_2 \cdots t_n \cdot (1/ t_i ). $$
Hence, by (\ref{eq:toricpara}), the relations among the initial monomials
are precisely given by $I_{\tilde T}$. We claim that each of the
quadratic binomial relations among the above Laurent monomials
lifts to exactly the same relation among the full polynomials $P_{ij}(t)$ and 
$P_{0i}(t)$. We shall prove this by examining the factorizations
of these polynomials.

In what follows we first assume that $T$ is a binary tree, i.e.~every
vertex in $V \backslash \{1,2,\ldots,n\}$ has precisely two children in $T$.
At the end of the proof, we  shall derive Theorem~\ref{thm:main}
for non-binary trees from the same statement for binary trees.

For any inner vertex $k$ in the rooted binary tree $T$, let $T_k$ denote the rooted tree obtained 
from $T$ by deleting all edges and vertices below $k$. Thus $T_k$
is a rooted tree with leaves $\{k\} \cup (\{1,\ldots,n\} \backslash {\rm de}(k))$.
Let $D_k(t)$ denote the determinant of its covariance matrix.
This is a homogeneous polynomial of degree $n+1-|{\rm de}(k)|$. For any directed edge $u \rightarrow v$ of the tree $T$, we consider the submatrix
of $\Sigma$ with row indices  ${\rm de}(u)\backslash {\rm de}(v)$ and
 column indices  $({\rm de}(u)\backslash {\rm de}(v)) \cup \{k\}$, for any
fixed $\,k \in {\rm de}(v)$. 
This matrix does not depend on $k$,
and it has one more column than
rows. We make it square by placing the
all-ones vector $(1,\ldots,1)$ into the first row.
 We write
$E_{u  v}(t)$ for the determinant of that square matrix.
This is a homogeneous polynomial in $(t_v)_{v\in V}$ of degree 
$|{\rm de}(u)\backslash {\rm de}(v)|$.
By convention,  $E_{0 v}=1$ for the root edge $0\to v$.

Consider the path between any two leaves $i$ and $j$ in the unrooted tree $\tilde T$. Each
vertex $u$ in the interior of such a path has a unique child $v$ in 
the rooted tree $T$ that is not on the path. 
Here we are using the assumption that $T$ is a binary tree.
 The only exception is the top vertex $u = {\rm lca}(i,j)$ 
 on the path between $i$ and $j$ in $T$. 
 
 We find that the polynomial
$P_{0i}(t)$ is equal to the product of all determinants $E_{u  v}(t)$
where $u \rightarrow v$ is any edge on the path from $0$ to $i$.
Similarly,  the $(n-1) \times (n-1)$ determinant $P_{ij}(t)$ 
is equal to $D_{{\rm lca}(i,j)}(t)$ times
the product of all $E_{u  v}(t)$
where the vertex $u \not= {\rm lca}(i,j)$ is on the path from leaf $i$ to leaf $j$.
One verifies this by examining for which parameter values $t$ these expressions  vanish,
and by noting that the initial monomials coincide with the
products of the initial monomials of the factors:
$$ 
\begin{matrix}
{\rm in}(D_k(t)) & = & t_k \cdot \prod \{ t_i \,: \, i \in \{1,\ldots,n\} \backslash {\rm de}(k)\},
\\
{\rm in}(E_{u  v}(t))  & = & \prod \{ t_i \,: \, i \in {\rm de}(v) \}.
\end{matrix}
$$

The above factorizations of $P_{0i}(t)$ and $P_{ij}(t)$ into the determinants
$D_\bullet(t)$ and $E_\bullet(t)$
 show that each generator 
$p_{ik} p_{jl} - p_{il} p_{jk} $ of $I_{\tilde T}$ vanishes on our variety.
By our analysis of the leading monomials, there are no relations
among the  polynomials $P_{ij}(t)$ and
$P_{0i}(t)$ beyond those in $I_{\tilde T}$.
In fact, our analysis shows that these polynomials  form a {\em Khovanskii  basis} 
(cf.~\cite{KM}) for the reverse lexicographic monomial order on the~$t_i$.

We now know that Theorem~\ref{thm:main} holds for all binary trees. It
remains to derive from this the same statement for all non-binary trees.
The property for rooted trees to be binary translates into the property
for unrooted trees to be trivalent.
Let $\tilde T$ be any non-trivalent tree and let $[\tilde T]$ 
be the set of all trivalent trees $\tilde U$ that are obtained by refining $\tilde T$.
One verifies that the following identity among toric ideals holds:
\begin{equation}
\label{eq:identityideals}
 I_\mathcal{\tilde T} \quad = \quad \sum_{\tilde U \in [\tilde T]}
I_{\tilde U}. 
\end{equation}

Similarly, the linear space $\mathcal{L}_T$ is the intersection of all
the linear spaces $\mathcal{L}_U$, where $\tilde U$ runs over
$[\tilde{T}]$. Since matrix inversion is a birational isomorphism,
the variety $\mathcal{L}_T^{-1}$ is the intersection of 
the toric varieties $\mathcal{L}_U^{-1}$ where $\tilde U $ runs over
the trivalent trees in $ [\tilde{T}]$. The Nullstellensatz implies that 
the sum of toric ideals (\ref{eq:identityideals})  cuts out 
$\mathcal{L}_T^{-1}$ set-theoretically. This shows that 
$\mathcal{L}_T^{-1}$ is a toric variety, with toric ideal in (\ref{eq:identityideals}).
  \end{proof}

\begin{example}\label{ex:running2}  Consider the binary tree in
Figure \ref{fig:running0} and Examples~\ref{ex:4k4} and~\ref{ex:M2}.
  The special determinants defined above are the following polynomials:
$$ 
E_{51} = t_2 - t_5 \, , \,\,
E_{52} = t_1 - t_5 \, , \,\,
E_{63} = t_4- t_6 \, , \,\,
E_{64} = t_3 - t_6\, , $$
$$ \begin{small}
E_{7  5} \,\,= \,\, {\rm det}
\begin{bmatrix}
 1&1&1 \\
t_7 &   t_3 &  t_6  \\
t_7 &   t_6&   t_4 
\end{bmatrix}, \quad
E_{7  6} \,\,= \,\,{\rm det}
\begin{bmatrix}
 1&1&1 \\
t_1 &   t_5 &  t_7  \\
t_5 &   t_2&   t_7 
\end{bmatrix}. \end{small}
$$

We are interested in the projective variety  in
$\mathbb{P}^{9}$ that is parametrized by
$$ \begin{matrix}
  p_{01} \, = \,  E_{75 }E_{51},\,\,
  p_{02} \, = \,  E_{75 } E_{52},\,\,
  p_{03}\, = \,  E_{76} E_{63},\,\,
  p_{04} \, = \, E_{76 } E_{64} ,\,\,
 p_{12}  \, = \, D_5,\,\,p_{34}  \, = \, D_6, \\
 p_{13}  \, = \,  E_{51} D_7 E_{63},\,\,\,
 p_{14}  \, = \,  E_{51} D_7 E_{64},\,\,\,
 p_{23}  \, = \,  E_{52} D_7 E_{63},\,\,\,
 p_{24}  \, = \,  E_{52} D_7 E_{64}.
 \end{matrix}
$$
One verifies that this is the variety defined by the toric ideal
$I_{\tilde T}$ seen in Example~\ref{ex:4k4}.
Furthermore, the same toric variety is also parametrized by the initial monomials
$ {\rm in}(p_{01})= t_2 t_3 t_4, \, {\rm in}(p_{02})= t_1 t_3 t_4, \,\ldots \, ,\,
{\rm in}(p_{24}) = t_1 t_3 t_7$, and $\,
{\rm in}(p_{34}) = t_1t_2 t_6 $. \hfill\qed
\end{example}

\begin{rem}
Tropical geometers know that the toric ideals
$I_{\tilde T}$ are precisely the monomial-free initial ideals of
the Pl\"ucker ideal that defines the Grassmannian of~lines. The latter
arises in a manner that is similar to our passage from
covariance matrices to concentration matrices, namely by 
inverting matrices $\Sigma$ that have a Hankel structure.
This is the content of \cite[Proposition 7.2]{MSUZ}. We do not
know whether this is related to the present paper. 
Is it possible to derive Theorem \ref{thm:main} by a degeneration argument from
the relationship between Hankel matrices $\Sigma$ and B\'ezout matrices $K$?
\end{rem}

\section{Maximum likelihood algebra}
\label{sec:mle}

The log-likelihood function for Gaussian random variables
is the function $\ell(\Sigma)$ in (\ref{eq:loglike1}).
Here $S = (s_{ij})$ is a fixed sample covariance matrix, i.e.~$S = \frac{1}{N} X X^T$ 
where $X$ is a real $n \times N$ matrix
whose columns are the observed samples.
Maximum likelihood estimation is concerned with
maximizing the expression (\ref{eq:loglike1})
over all covariance matrices $\Sigma = (\sigma_{ij})$ in the model of interest.
This optimization problem is equivalent to maximizing the 
expression (\ref{eq:loglike2}) over all concentration matrices $K$ in the model.

The optimal solution to this problem is denoted
by $\hat \Sigma = ({\hat \sigma}_{ij})$ or
$\hat K = ({\hat \kappa}_{ij})$. This is called the
{\em maximum likelihood estimate} (MLE) for the data $S$.
Here the model is fixed but the data $S$ can vary. 
We therefore think of the MLE as a function of $S$.

In this section we study the MLE
for the Brownian motion tree model
$\mathcal{L}_{T,\geq}$. The idea is to take advantage of
the toric structure revealed in Theorem~\ref{thm:main}. Thus,
we use the coordinate change (\ref{eq:pcoord})
that writes the concentration matrix $K$ 
as the reduced Laplacian for the complete graph
on $n+1$ vertices with edge labels $p_{ij}$.
With this, the expression (\ref{eq:loglike2})  is a function
of the $p_{ij}$, subject to the toric constraints in $I_{\tilde T}$.
This gives us the flexibility to choose a convenient parametrization
of the toric variety.

In algebraic statistics, one distinguishes two kinds of polynomial constraints
for a statistical model, namely equations and inequalities.
It is customary to  first focus on the equations and examine the MLE in that setting before incorporating inequalities.

In our paper, the model
is given by the semialgebraic set $\mathcal{L}_{T,\geq}^{-1}$. 
This set satisfies the inequalities in Theorem~\ref{thm:mainsemi}.
For the discussion of MLE in the current section, we ignore the inequality constraints and identify the set $\mathcal{L}_{T,\geq}^{-1}$ with its
Zariski closure, which is the toric variety $\mathcal{L}_T^{-1} = \mathcal{V}(I_{\tilde T})$.
The critical points of the likelihood function $\ell(K)$ on that variety
are defined by a system of polynomial equations, 
known in statistics as the {\em likelihood equations}.
These can be derived by using Lagrange multipliers, or
via a monomial parametrization of
the toric variety $\mathcal{V}(I_{\tilde T})$.

The {\em maximum likelihood degree} of the model is, by definition,
the number of complex solutions to the likelihood equations for generic data $S$.
This number is an algebraic invariant of the ideal $I_{\tilde T}$. To compute it
we take $S$ to be a general symmetric $n \times n$ matrix of full rank $n$
and we count all complex critical points of $\ell(K)$.
 
\begin{prop} \label{prop:mldegree}
The maximum likelihood degree
of the Brownian motion tree model on a binary tree $\tilde T$ with
$\,n=2,3,4,5,6,7,8\,$ leaves is equal to
$\,1,1,5,17,61,233,917$.
\end{prop}

\begin{proof} This result was found by
symbolic computation, namely using the Gr\"obner basis
package  in the computer algebra system {\tt maple}.  For $n \geq 6$ the computation
was carried out over a finite field. All combinatorial types of trees were considered.
See Example \ref{ref:galois5} for an illustration
of the case  $n=4$ where the ML degree is $5$.
\end{proof}

This result is complementary to the usual approach in computational statistics
where one maximizes the likelihood function using a local numerical method,
such as the Newton-Raphson algorithm. Local methods perform best in a regime
where the likelihood function is concave. Such a regime was identified  in
 \cite{ZUR14}, where concavity was shown to hold with high probability when
 the dimension $n$ is small relative to the sample size $N$. In that analysis it was
 essential to use all constraints of the model, i.e., not just
 the equations but also the inequalities. 
 
The maximum likelihood degree being equal to one  means that the MLE
can be written as a rational function of the data. Proposition \ref{prop:mldegree}
says that this happens for our model when $n =2$ and $n=3$. We next present 
the formulas for these two cases.
 
 \begin{example}[$n=2$]
The toric ideal $I_{\tilde T}$ equals $\{0\}$, so our model is the full
Gaussian family. This means that the MLE equals the sample covariance matrix:
$$ \hat \sigma_{11} = s_{11},\,\,
 \hat \sigma_{12} = s_{12},\,\,  \hat \sigma_{22} = s_{22}. $$
Since the MLE of the parameters is $\hat t_1=s_{11}$, $\hat t_2=s_{22}$, $\hat t_3=s_{12}$, this leads to valid parameters for the Brownian motion tree model if $\min\{s_{11},s_{22}\}\geq s_{12}\geq 0$.\hfill\qed
 \end{example}
 
 \begin{example}[$n=3$]
 We label the rooted tree $T$ so that
 $\{1,2\}$ is a clade. Hence  $\{1,2\}$ and $\{0,3\}$ are the
 cherries in the unrooted tree $\tilde T$. Our toric ideal is principal:
  $$ I_{\tilde T} \,=\, 
 \langle p_{01} p_{23} - p_{02} p_{13} \rangle \,=\,
 \langle
 \kappa_{11} \kappa_{23}- \kappa_{12} \kappa_{13} 
 + \kappa_{12} \kappa_{23} - \kappa_{13} \kappa_{22}
 \rangle.
 $$
This is equivalent to setting $\sigma_{13} = \sigma_{23}$ in the
covariance matrix $\Sigma = K^{-1}$.
The MLE  is a rational function of the entries $s_{ij}$ of the sample covariance 
matrix $S$. We define
$$
c\,=\,(s_{11}-2 s_{12}+s_{22})s_{33} - (s_{13}-s_{23})^2.
$$
The entries $\hat \sigma_{ij}$ of the estimated
covariance matrix  $\hat\Sigma$ satisfy
$\, \hat\sigma_{33}  =   s_{33}\,$ and
$$ \begin{matrix}
\hat\sigma_{11} &\! = \! &  s_{11} - 2 (s_{13}{-}s_{23}) (s_{11} s_{33}{-}s_{12} s_{33}
{-}s_{13}^2{+}s_{13} s_{23}) (s_{11} s_{23}{-}s_{12} s_{13}{-}s_{12} s_{23}{+}s_{13} s_{22})/c^2 ,\\
\hat\sigma_{12} & \! = \! &  s_{12} - \, (s_{13}{-}s_{23}) (s_{11} s_{33}{-}s_{13}^2{-}s_{22} s_{33}
{+}s_{23}^2) (s_{11} s_{23}-s_{12} s_{13}-s_{12} s_{23}+s_{13} s_{22})/c^2 ,\\
\hat\sigma_{22} & \! =\!&  s_{22} - 2 (s_{13}{-}s_{23}) (s_{12} s_{33}{-}s_{13} s_{23}
{-}s_{22} s_{33} {+}s_{23}^2) (s_{11} s_{23}{-}s_{12} s_{13}{-}s_{12} s_{23}{+}s_{13} s_{22})/c^2.
\end{matrix}
$$

The remaining two matrix entries must be equal:
$$
\begin{matrix}
& \hat\sigma_{13} & =& 
 s_{13} - (s_{13}-s_{23}) (s_{11} s_{33} -s_{13}^2-s_{12} s_{33}+s_{13} s_{23})/c
 \\
= & \hat\sigma_{23} &  = & s_{23} - (s_{23}-s_{13}) 
(s_{22} s_{33}-s_{23}^2-s_{12} s_{33}+s_{13} s_{23})/c.
\end{matrix}
$$
The following two linear forms are preserved
when passing from data to MLE:
$$
\hat \sigma_{11}-2\hat \sigma_{12}+\hat \sigma_{22}\;=\;s_{11}-2s_{12} +s_{22}
\qquad {\rm and} \qquad \hat\sigma_{33}  =   s_{33}.
$$
Writing $K = (k_{ij}) = S^{-1}$ for the sample concentration matrix, we 
note that $K-\hat K$ is a rank $2$ matrix which depends only on $s_{33}$, $s_{13}-s_{23}$, and $s_{11}-2s_{12}+s_{22}$. Also,
$ c \;= \; (k_{11} + 2 k_{12} + k_{22})/{\rm det}(K).$ \hfill\qed
\end{example} 

\begin{example}[$n=4$] \label{ref:galois5}
We consider the  tree $\tilde T$ in Figure \ref{fig:running0}.
Its toric variety $\mathcal{V}(I_{\tilde T}) = \mathcal{L}_T^{-1} \subset \mathbb{P}^9$
 was discussed in Examples 
\ref{ex:4k4}, \ref{ex:M2} and \ref{ex:running2}. 
We shall prove that the MLE for this model cannot be expressed in radicals.
For this, we fix the parametrization
\begin{equation}
\label{eq:nicepara}
 \begin{matrix} p_{01} = u_1,\, p_{02} = u_2,\, p_{03} = u_3,\, p_{04} = u_4,\,\,
 p_{12} = u_1 u_2 u_6,\, p_{13} = u_1 u_3 u_5, \,\\ \qquad \quad
    p_{14} = u_1 u_4 u_5,\, \,\, p_{23} = u_2 u_3 u_5,\,\,\,
     p_{24} = u_2 u_4 u_5,\,\,\, p_{34} = u_3 u_4 u_7 . 
\end{matrix}     
\end{equation}
We substitute this into the  concentration matrix $K$ in Example \ref{ex:4k4}.
The determinant of that $4 \times 4$ matrix   is a polynomial of
degree $10$ with $81$ terms:
$$ {\rm det}(K) \, = \,
    u_1^4 u_2 u_3 u_4 u_5^2 u_6+3 u_1^3 u_2^2 u_3 u_4 u_5^2 u_6
    +u_1^3 u_2 u_3^2 u_4 u_5^3
+ \,\cdots \,    +u_1 u_2 u_3 u_4^2 u_7+u_1 u_2 u_3 u_4.
$$

For our computation we now take the  sample covariance matrix 
\begin{equation}
\label{eq:dataS}
\begin{small}
S \,\, = \,\, \begin{bmatrix}
5 & 3 & 1 & 2 \\
 3 & 5 & 1 & 1 \\
 1  & 1 & 5 & 3 \\
  2  & 1 & 3 & 4 \\
 \end{bmatrix} .\end{small}
 \end{equation}
 Thus $ {\rm trace}(SK) =   4 u_1 u_2 u_6 + 8 u_1 u_3 u_5 +  \cdots  + 5 u_3 + 4 u_4 $.
 Our goal is to maximize the likelihood function 
 $\log(\det( K))-{\rm trace}(SK) $ where $(u_1,u_2,\ldots,u_7)$ ranges over $\RR^7$.
Its seven partial derivatives are rational functions in the $u_j$. We clear
denominators and impose ${\rm det}(K) \not= 0$.
This results in a system of polynomial equations. We fix
the lexicographic term order with $u_1 > u_2 > \cdots > u_7$,
 we compute the reduced Gr\"obner basis in {\tt maple}, and we find
 that it has a triangular shape. For $i=1,2,\ldots,6$, the Gr\"obner basis 
 has an element $u_i - p_i(u_7)$,
where $p_i$ is a univariate polynomial of degree six with
large rational coefficients. In addition, we see the quintic polynomial
\begin{equation}
\label{eq:S5}
\begin{small} \quad
\begin{matrix}
5955844829180400 \,u_7^5 \,-\, 203897411425749580 \, u_7^4  \,+\, 129689372089999498 \,u_7^3 \\
     - \,139971736881354888 \,u_7^2  \,+\, 44907572962723196 \,u_7 \,-\, 5517030143672333.
     \end{matrix}
     \end{small}
\end{equation}
This polynomial has precisely one real root at
$\hat u_7 = 33.607528... $. By back-substitution, we compute
the estimated concentration matrix $\widehat K$, and we find its inverse to be
$$ \begin{small} \widehat \Sigma \, = 
\begin{bmatrix}
  4.757115029565996  & \!3.040016668717226 & \! 1.418803877886187 &
  \! 1.418803877886187 \\
  3.040016668717226 &  \! 5.322918307868457 & \! 1.418803877886187 & 
  \! 1.418803877886187  \\
  1.418803877886187 & \! 1.418803877886187 & \! 5.295621559259030 & 
  \! 3.094192269251272\\
  1.418803877886187 & \! 1.418803877886187 & \! 3.094192269251272 & 
  \! 3.892762979243514
\end{bmatrix}\!\!.
\end{small}
$$
Note that this matrix lies in $\mathcal{L}_{T,\leq}$, so it is the true MLE for the statistical model. 

Using {\tt maple}, we also check that the {\em Galois group} of the polynomial (\ref{eq:S5})
over the rational numbers $\mathbb{Q}$ is the symmetric group on $5$ letters.
Hence $\hat u_7$ cannot be written in radicals over  $\mathbb{Q}$.
This implies that the MLE  cannot be written in radicals. \hfill\qed
\end{example}

Proposition \ref{prop:mldegree} only applies to rooted trees $T$
that are binary. This raises the question what happens for
 degenerate tree topologies. At first glance, one might
think that the ML degree decreases for special trees.
However, this is not the case:

\begin{example}[$n=4$ revisited] \label{ex:revisit4}
Let $\tilde T$ be the star tree on five leaves, so
$T$ is the directed tree obtained from the tree in Figure~\ref{fig:running0}
by shrinking the edges labeled $\theta_5 $ and $\theta_6$.
Recall from Example \ref{ex:M2} that $I_{\tilde T}$ has
codimension $5$ and degree $11$. It has  $10$ quadratic generators.
We obtain a parametrization by setting $u_5 = u_6 = u_7$ in
(\ref{eq:nicepara}). 

Performing the same computation as in
Example \ref{ref:galois5}, we find that the ML degree of this 
star tree model is {\bf 21}. For the sample covariance matrix $S$ in (\ref{eq:dataS}), we find
$$ \begin{small}
\widehat \Sigma \, = \begin{bmatrix}
        2.945585253871356    &
         3.350776006974025     &
         3.350776006974025    &
         3.350776006974025 \\
  \!       3.350776006974025     &
  \!       5.654229276230780   &
  \!     3.350776006974025    &
  \!     3.350776006974025 \\
   \!     3.350776006974025     &
   \!    3.350776006974025    &
   \!  \!  11.10861203686456 &
   \!    3.350776006974025  \\
   \!    3.350776006974025     &
   \!    3.350776006974025    &
   \!    3.350776006974025  &
   \!    5.509927633167128
                  \end{bmatrix} \!\! .
\end{small}
$$
The matrix entries in $\widehat{\Sigma}$ are algebraic numbers of degree {\bf 21} over $\mathbb{Q}$. \hfill\qed
\end{example}

In the case of star trees, the MLE problem can be formulated via  (\ref{eq:loglike1}) as follows:

\begin{itemize}
\item[$\bullet$] {\em
Minimize $\,\log(\det \Sigma)-{\rm trace}(S \Sigma^{-1})\,$
over the set of symmetric matrices $\Sigma \in \mathbb{S}^n_+$ whose off-diagonal entries are equal and smaller than the diagonal entries.}
\end{itemize}
We obtained the following result concerning the algebraic degree 
of this optimization problem.
Just like Proposition \ref{prop:mldegree},
this was found using computations with {\tt maple}.

\begin{prop} \label{prop:mldstar}
The maximum likelihood degree
of the Brownian motion star tree model with $\,n=2,3,4,5,6,7,8,9\,$ is equal to
$\,\delta_n= 1, 7, 21, 51, 113, 239, 493, 1003$.
\end{prop}

It is natural to conjecture that this degree  always satisfies
$\,\delta_{n+1} = 2(\delta_n+n)+1$.

\begin{rem} The estimated matrix $\widehat \Sigma$
in Example \ref{ex:revisit4} lies in the spectrahedron
$\mathcal{L}_{T} \cap \mathbb{S}^4_+$. It is {\bf not}
 in the model $\mathcal{L}_{T,\geq}$ for the star tree $\tilde T$
 because the upper left entry is
smaller than the off-diagonal entry in the first row.
This discrepancy  motivates studying
the inequalities in Theorem~\ref{thm:mainsemi}, whose proof is given in
the next ~section.
\end{rem}

\section{Being on trek in semialgebraic statistics}\label{sec:semialg}

Our goal is to prove that the inequalities  in Theorem~\ref{thm:mainsemi} are
valid for our model. Namely, we show that
 $p_{ij}$ and $p_{ij} p_{kl} - p_{il} p_{jk}$ are nonnegative on  $\mathcal{L}_{T,\geq}$.
This is done by applying the theory
of treks due to Sullivant, Talaska and Draisma \cite{DST0}.

A symmetric $n \times n$ matrix
$K$ is an {\em M-matrix} if $K$ is positive definite and $\kappa_{ij}\leq 0$ for all $i\neq j$. Moreover, $K$ is {\em diagonally dominant} if 
$|\kappa_{ii}|\geq \sum_{j\neq i}|\kappa_{ij}|$ for all $i$. If $K$ is an 
M-matrix then it is diagonally dominant if and only if the vector $K\mathbf 1$ 
has nonnegative entries. Therefore, a matrix $K=[\kappa_{ij}]$ is 
a diagonally dominant M-matrix if   and only $K \in \mathbb{S}^n_+$ and the
quantities $p_{ij}=-\kappa_{ij}$ and $p_{0i}=\sum_{j=1}^n \kappa_{ij}$ in (\ref{eq:pcoord})
are nonnegative.

It is known in linear algebra \cite[Theorem 2.2]{varga1993symmetric}
 that the inverse of any symmetric ultrametric matrix is a
diagonally dominant M-matrix. This explains why all points in
$\mathcal L_{T,\geq}^{-1}$  have  nonnegative coordinates.
This constraint is the first in (\ref{eq:constraints}).
The validity of the other inequality constraints
arises from the following key lemma.

\begin{lem} \label{lem:key} The determinant
 ${\rm det}(\Sigma)$ times
the quantity $\,p_{ij}p_{kl} - p_{il}p_{jk}\,$ in (\ref{eq:constraints})
is a sum of products of parameters $\theta_i$, so it is nonnegative
when the $\theta_i$ are nonnegative.
\end{lem}

The proof of this lemma is given below.
The case $n=4$ was seen in Example~\ref{ex:5_15}. To provide some intuition, we now prove 
Lemma~\ref{lem:key} and Theorem~\ref{thm:mainsemi} for 
$n\leq 3$.
 
 \begin{example}[$n \leq 3$]
 \label{ex:mainsemismall}
 Let $n=2$. There are no constraints in (\ref{eq:constraints}), and we have
 $$	\begin{small} K \, =\, \Sigma^{-1}\;\;=\;\;\begin{bmatrix}
		\theta_1+\theta_3 & \theta_3\\
		\theta_3 & \theta_2+\theta_3
	\end{bmatrix}^{-1}\;\;=\;\;\frac{1}{\theta_1\theta_2+\theta_1\theta_3+\theta_2\theta_3}\begin{bmatrix}
		\theta_2+\theta_3 & -\theta_3\\
		-\theta_3 & \theta_1+\theta_3
	\end{bmatrix}. \end{small}
	$$
	Assuming that $K$ lies in $\mathbb{S}^2_+$, then the vector 
	$(p_{01},p_{02},p_{03})$ is nonnegative if and only if
$(\theta_1,\theta_2,\theta_3) = {\rm det}(\Sigma){\cdot} (p_{02},p_{01},p_{12})$ is nonnegative.
This proves Theorem~\ref{thm:mainsemi} for~$n=2$.

Let $n=3$ and $T$ be the binary tree with clade $\{1,2\}$. Theorem~\ref{thm:mainsemi}
asserts that the model
  $\mathcal L_{T,\geq}^{-1}$ is equal to the set of all diagonally dominant M-matrices satisfying
\begin{equation}\label{eq:extra3}
p_{01}p_{23}\;=\;p_{02}p_{13}\;\leq \;p_{03}p_{12}.	
\end{equation}
The former is contained in the latter because a direct calculation reveals that
\begin{equation}
	\begin{aligned}\label{eq:eqs3}
		p_{01}\det\Sigma=\theta_2\theta_3,\quad p_{02}\det\Sigma=\theta_1\theta_3,\quad p_{03}\det\Sigma=\theta_1\theta_2+\theta_1\theta_4+\theta_2\theta_4,\\
p_{12}\det\Sigma=\theta_3\theta_4+\theta_3\theta_5+\theta_4\theta_5,\quad p_{13}\det\Sigma=\theta_2\theta_5,\quad p_{23}\det\Sigma=\theta_1\theta_5, \\
(p_{03}p_{12}-p_{01}p_{23})\det\Sigma\;=\;\theta_4.
\end{aligned}
\end{equation}
Conversely, let $K$ be a diagonally dominant M-matrix satisfying (\ref{eq:extra3}). By Theorem~\ref{thm:main}, the equation in (\ref{eq:extra3}) implies that $K\in \mathcal L_T^{-1}$. Since $K$ is invertible, we can define $\Sigma=K^{-1}$.
Then $\Sigma=\Sigma_\theta$ for some real vector $(\theta_1,\ldots,\theta_5)$ 
that satisfies (\ref{eq:eqs3}).
From (\ref{eq:extra3}) we obtain $\theta_4\geq 0$. Nonnegativity of $p_{0i},p_{ij}$ implies that $\theta_1,\theta_2,\theta_3,\theta_5$ are either all
nonpositive or all nonnegative. We want to show that they are all nonnegative. Suppose  they are negative. Since $p_{03}\det \Sigma\geq 0$ and $p_{12}\det \Sigma\geq 0$, 
we~have
$$
\theta_4\;\leq\;-\frac{\theta_1\theta_2}{\theta_1+\theta_2},\qquad \theta_4\;\leq\;-\frac{\theta_3\theta_5}{\theta_3+\theta_5}.
$$
However,  $\,
\det \Sigma\;=\;\theta_4(\theta_1+\theta_2)(\theta_3+\theta_5)+\theta_1\theta_2(\theta_3+\theta_5)+(\theta_1+\theta_2)\theta_3\theta_5\;> \; 0\,$
and so 
$$
\theta_4\;>\;-\frac{\theta_1\theta_2}{\theta_1+\theta_2}-\frac{\theta_3\theta_5}{\theta_3+\theta_5}.
$$
This is a contradiction and hence Theorem~\ref{thm:mainsemi} holds for $n \leq 3$.~\qed
 \end{example}

Fix the tree $T$ with $n$ leaves as before.
A \emph{trek} from leaf $i$ to leaf $j$ is a pair $\gamma=(P_L,P_R)$, where $P_L$ is a directed path from some vertex $v$ to $i$ and $P_R$ is a directed path from $v$ to $j$. The leaf $i$ is the {\em initial vertex}, the leaf $j$ is the {\em final vertex}, and 
$v=v(\gamma)$ is the {\em top} of the trek. The parameter
 $\theta_{v(\gamma)}$ is the  weight of the trek.
We also allow treks between $i=0$ and $j$, in which case the associated weight is $1$. Given two sets $A$ and $B$ with the same cardinality, a \emph{trek system} $\Gamma$ from $A$ to $B$ consists of $|A|$ treks whose initial vertices exhaust the set $A$ and whose final vertices exhaust the set $B$. 
The weight of a trek system is the product of the weights of all its treks.

In our application either $A=\{1,\ldots,n\}\backslash \{j\}$ and $B=\{1,\ldots,n\}\backslash \{i\}$ if $1\leq i\leq j\leq n$, or $A=\{0,1,\ldots,n\}\backslash \{j\}$ and $B=\{1,\ldots,n\}$ if $0=i<j\leq n$. 
We assume that all treks are mutually vertex-disjoint. Equivalently, we consider
the set $\mathcal T_{i,j}$ of trek systems $\Gamma$ from $A$ to $B$ that consist of the following $|A|$ vertex-disjoint treks:
\begin{enumerate}
		\item[(i)] one trek from $i$ to $j$,
		\item[(ii)] $|A\cap B|$ treks from $k$ to $k$ for each $k\in A\cap B$.
\end{enumerate} 
In Figure~\ref{fig:4treks} we display the set $\mathcal{T}_{1,2}$ for 
our running example. Its elements are the
eight trek systems from $A=\{1,3,4\}$ to $B=\{2,3,4\}$ in the tree $T$
shown in Figure~\ref{fig:running0}.

\begin{figure}[b]
	\includegraphics[scale=.86]{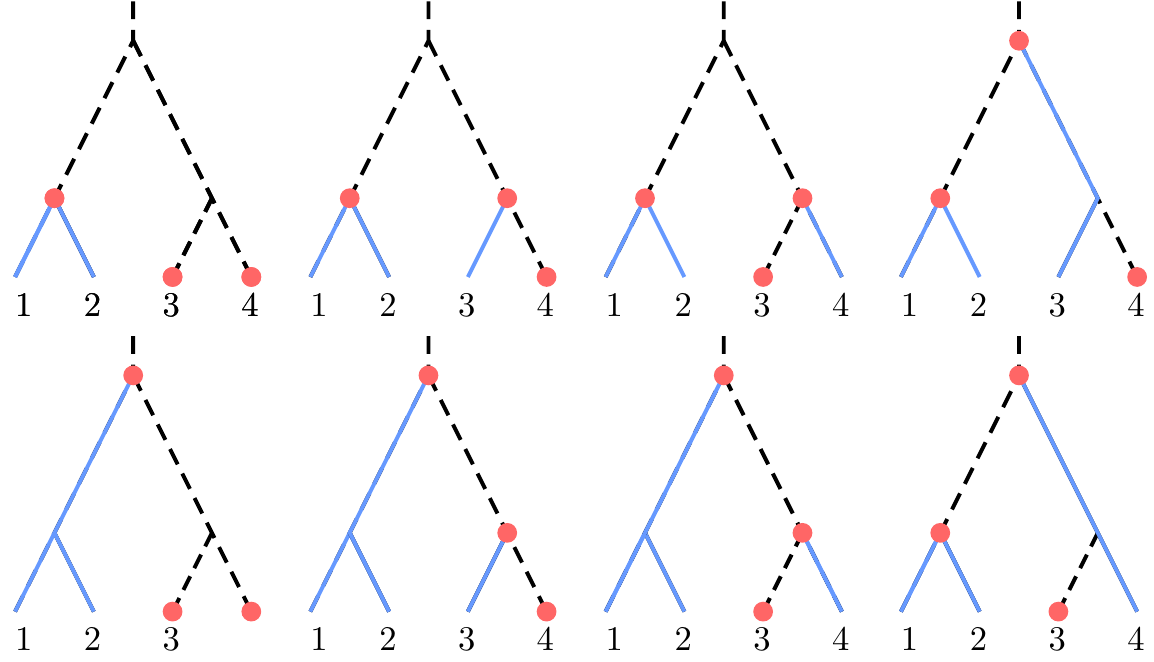}
	\caption{Eight trek systems from $A=\{1,3,4\}$ to $B=\{2,3,4\}$. Treks are indicated by solid edges. Dots mark the tops of the treks.}\label{fig:4treks}
\end{figure}

 \begin{prop}\label{prop:pijtheta}
 The following identity holds for all	indices $\,0\leq i<j\leq n$:
$$
p_{ij}\det \Sigma\;=\;\sum_{\Gamma\in \mathcal T_{i,j}} \prod_{\gamma\in \Gamma}\theta_{v(\gamma)}
$$
Moreover, each monomial appears in this sum only once.
\end{prop}

\begin{proof}
We first prove the second assertion: each trek system in $\mathcal T_{i,j}$ gives a different monomial.
 Suppose there are two different trek systems $\Gamma$, $\Gamma'$ such that $\prod_{\gamma\in \Gamma}\theta_{v(\gamma)}=\prod_{\gamma\in \Gamma'}\theta_{v(\gamma)}$. Let $\gamma,\gamma'$ be two treks in $\Gamma,\Gamma'$ with the initial point $k$ and the final point~$l$ (either $k=l$ or $k=i$, $l=j$). Both $v(\gamma)$ and $v(\gamma')$ lie on the path between $0$ and ${\rm lca}(k,l)$. If $v(\gamma)\neq v(\gamma')$ then one lies above the other, say  $v(\gamma)< v(\gamma')$. But then $v(\gamma')$ lies on the trek of $\Gamma$ from $k$ to $l$ and so it cannot be at the top of another trek in $\Gamma$ because treks must be mutually disjoint. This leads to a contradiction unless $\gamma=\gamma'$. We conclude that~$\Gamma=\Gamma'$.
 
We now prove the main formula. Consider first the case $1\leq i \leq j\leq n$. 
We have
$$
\kappa_{ij} \det \Sigma\;=\;(-1)^{i+j}\det \Sigma_{A,B},
$$
where $A=\{1,\ldots,n\}\backslash \{j\}$ and $B=\{1,\ldots,n\}\backslash \{i\}$. Let $\Lambda=[\lambda_{uv}]\in \{0,1\}^{V\times V}$ be the matrix with 
$\lambda_{uv}=1$ if $u\to v$ in $T$ and $\lambda_{uv}=0$ otherwise, and
 let $D_\theta$ be the diagonal matrix with entries $\theta=(\theta_v)$.
  The covariance matrix of the model (\ref{eq:BMT}) equals
$$
\Xi\;\;=\;\;(I-\Lambda)^{-T}D_\theta(I-\Lambda)^{-1}.
$$ 
The principal submatrix of $\Xi$ corresponding to the leaves of $T$ is $\Sigma$. Hence $\det\Sigma_{A,B}=\det \Xi_{A,B}$. Every trek system between $A$ and $B$ gives rise to a permutation $\pi\in S_{n-1}$ and we define ${\rm sign}(\Gamma):={\rm sign}(\pi)$. If $\Gamma\in \mathcal T_{i,j}$ then ${\rm sign}(\pi)=(-1)^{i+j+1}$ unless $i=j$ in which case $\pi$ is the identity. Using equation (2) in \cite{DST},  we conclude that 
$$
\det \Sigma_{A,B}\;=\;\begin{cases}
\phantom{-} \sum_{\Gamma\in \mathcal T_{i,i}}\prod_{\gamma\in \Gamma}\theta_{v(\gamma)} & 
\mbox{if }i=j ,\\
	-\sum_{\Gamma\in \mathcal T_{i,j}}(-1)^{i+j}\prod_{\gamma\in \Gamma}\theta_{v(\gamma)} & \mbox{if } i\neq j.
\end{cases}
$$
The formula in \cite{DST} involves all trek systems between $A$ and $B$, but the sum can be restricted to trek systems with no sided intersections \cite[Definition 3.2]{DST}. In our case 
this is equivalent to treks being mutually vertex-disjoint. It follows that
$$
\kappa_{ij}\det \Sigma\;=\;\begin{cases} \phantom{-} 
\sum_{\Gamma\in \mathcal T_{i,i}} \prod_{\gamma\in \Gamma}\theta_{v(\gamma)} & \mbox{if }i=j , \\
-\sum_{\Gamma\in \mathcal T_{i,j}} \prod_{\gamma\in \Gamma}\theta_{v(\gamma)} & \mbox{if } i\neq j.\end{cases}
$$
This proves the desired formula for $p_{ij}\det \Sigma$ in the case when $i\neq 0$.

It remains to consider the case $i=0$. Here we have
$$
p_{0j}\det \Sigma\;=\;(\kappa_{jj}+\sum_{k\neq j} \kappa_{jk})\det \Sigma\;=\; \sum_{\Gamma\in \mathcal T_{j,j}} \prod_{\gamma\in \Gamma}\theta_{v(\gamma)} -\sum_{k\neq j}\sum_{\Gamma\in \mathcal T_{j,k}} \prod_{\gamma\in \Gamma}\theta_{v(\gamma)}.$$
We claim that this expression  equals $\sum_{\Gamma\in \mathcal T_{0,j}} \prod_{\gamma\in \Gamma}\theta_{v(\gamma)}$. For a fixed $k$, pick $\Gamma\in \mathcal T_{j,k}$ with associated monomial $\prod_{\gamma\in \Gamma}\theta_{v(\gamma)}$. 
Replace the trek $(P_L,P_R)$  from $j$ to $k$ in $\Gamma$
with the trek $(P_R,P_R)$ from $j$ to $j$. 
 The resulting  trek system $\tilde \Gamma\in \mathcal T_{j,j}$ has  the same weight. This shows that  the monomials in the second sum all appear in the first sum. Since each monomial appears at most once in a trek system, they mutually cancel each other out. The only terms of the first sum that remain are the ones not containing $\theta_v$ for $v\leq j$. These are precisely the monomials in $\,\sum_{\Gamma\in \mathcal T_{0,j}} \prod_{\gamma\in \Gamma}\theta_{v(\gamma)}$.
\end{proof}

\begin{example}
	Fix the tree in Figure~\ref{fig:running0}. Proposition~\ref{prop:pijtheta} 
	confirms the formula for $p_{12}\det \Sigma$ in Example~\ref{ex:5_15}.
	There are eight vertex-disjoint trek systems from $A=\{1,3,4\}$ to $B=\{2,3,4\}$ as shown in Figure~\ref{fig:4treks}. The trek systems in the first row have
	the weights $\theta_3\theta_4\theta_5$, $\theta_4\theta_5\theta_6$, $\theta_3\theta_5\theta_6$,  $\theta_4\theta_5\theta_7$. In the second row we get $\theta_3\theta_4\theta_7$, $\theta_4\theta_6\theta_7$, $\theta_3\theta_6\theta_7$, $\theta_3\theta_5\theta_7$. The sum of these eight monomials equals $p_{12}\det\Sigma$.
	\hfill \qed
\end{example}

We shall now prove the key lemma that was stated at the beginning of this section.

\begin{proof}[Proof of Lemma \ref{lem:key}] Let $ij|kl$ be a trivalent quartet
in $\tilde T$. Our goal is to show that $p_{ij}p_{kl}-p_{ik}p_{jl}$ is
 a sum of products of the parameters $\theta_v$. 
 Let $s=\det \Sigma$.
 By Proposition~\ref{prop:pijtheta}, we have
\begin{equation}\label{eq:binposiaux}
(p_{ij}p_{kl}-p_{ik}p_{jl})s^2  = \sum_{\Gamma\in \mathcal T_{i,j}} \sum_{\Gamma'\in \mathcal T_{k,l}} \prod_{\gamma\in \Gamma}\theta_{v(\gamma)}
\prod_{\gamma\in \Gamma'}\! \theta_{v(\gamma)}-\sum_{\Gamma\in \mathcal T_{i,k}} \sum_{\Gamma'\in \mathcal T_{j,l}} \prod_{\gamma\in \Gamma}\theta_{v(\gamma)}
\prod_{\gamma\in \Gamma'}\! \theta_{v(\gamma)}.	 \!\!
\end{equation}
It suffices to show that each term in the right sum lies also in the left sum. Fix a pair $\Gamma_{ik}\in \mathcal T_{i,k}$, $\Gamma_{jl}\in \mathcal T_{j,l}$.  We will construct  trek systems $\Gamma_{ij}\in \mathcal T_{i,j}$, $\Gamma_{kl}\in \mathcal T_{k,l}$ such~that 
\begin{equation}\label{eq:auxbinom}
\prod_{\gamma\in \Gamma_{ik}}\theta_{v(\gamma)}
\prod_{\gamma\in \Gamma_{jl}}\theta_{v(\gamma)}\;\;=\;\;\prod_{\gamma\in \Gamma_{ij}}\theta_{v(\gamma)}
\prod_{\gamma\in \Gamma_{kl}}\theta_{v(\gamma)}.	
\end{equation}
The idea of the construction is shown in  Figure~\ref{fig:ineqthm}.

Since $ij|kl$ is a trivalent quartet in $\tilde T$, either $v(\gamma_{ik})\leq j$ or $v(\gamma_{jl})\leq i$.
Otherwise the paths $\overline{ik}$ and $\overline{jl}$ do not intersect. Similarly, either $v(\gamma_{ik})\leq l$ or $v(\gamma_{jl})\leq k$. Without loss of generality, we 
consider the case $v(\gamma_{ik})\leq j$ and $v(\gamma_{ik})\leq l$.
(The proof is similar for the other three cases). Replace $P_R$ in $\gamma_{ik}=(P_L,P_R)$ with a path from $v(\gamma_{ik})$ to $j$ to obtain trek $\gamma_{ij}$ from $i$ to $j$. Replace $P_L'$ in $\gamma_{jl}=(P_L',P_R')$ with a path from $v(\gamma_{jl})$ to $k$ to obtain trek $\gamma_{kl}$ from $k$ to $l$. 
The quartet $ij|kl$ has two inner vertices $u,v$. 
Removing the path $\overline{uv}$ between $u$ and $v$ in $T$ together with all incident edges induces a split of $T$ into $\geq 4$ blocks. 
Since $\overline{uv}$ appears in $\gamma_{ik}$ and $\gamma_{jl}$, it cannot be a part of any other trek in $ \Gamma_{ik},  \Gamma_{jl}$. Therefore, all treks in both trek systems (apart from $\gamma_{ik}, \gamma_{jl}$) are entirely contained in one of the $\geq 4$ blocks. Denote the blocks containing $j,k$ by $A_j, A_k$, respectively. Let $\Gamma_{ij}$ be the trek system obtained from $\Gamma_{ik}$ by replacing $\gamma_{ik}$ with $\gamma_{ij}$ and all treks in $A_j\cup A_k$ with the treks of $\Gamma_{jl}$ contained in $A_j\cup A_k$. Similarly, let $\Gamma_{kl}$ be the trek system obtained from $\Gamma_{jl}$ by replacing $\gamma_{jl}$ with $\gamma_{kl}$ and all treks in $A_j\cup A_k$ with the treks of $\Gamma_{ik}$ contained in $A_j\cup A_k$.

\begin{figure}[t]
	\includegraphics[scale=.7]{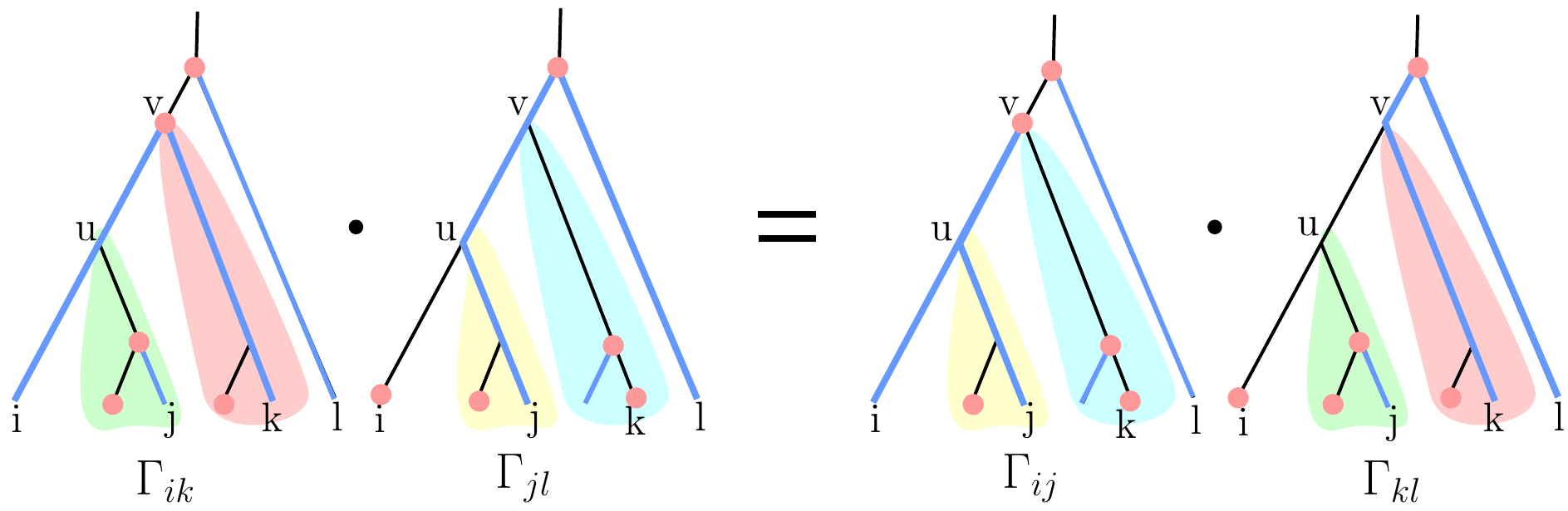}
	\vspace{-0.12in}
	\caption{Illustration of the construction for (\ref{eq:auxbinom}) in the proof of Lemma~\ref{lem:key}.}\label{fig:ineqthm}
\end{figure}

By construction, the power of $\theta_v$ coincides on both sides of (\ref{eq:auxbinom}) and so the corresponding terms in (\ref{eq:binposiaux}) will cancel out. What is left  is a sum of weights of trek systems, 
and hence  a sum of products of parameters $\theta_i$.
\end{proof}

\begin{rem}
A similar construction, also based on Proposition~\ref{prop:pijtheta},
can be used to show that the terms in  $p_{ik}p_{jl} s^2$ are
precisely equal to the terms in $p_{il}p_{jk}s^2 $.
This gives an alternative proof of the equations in  (\ref{eq:constraints})
and hence of Theorem \ref{thm:main}.
\end{rem}


We now prove the semialgebraic characterization
of Brownian motion tree models.

\begin{proof}[Proof of Theorem \ref{thm:mainsemi}]
We first claim that it suffices to show the result for binary trees.
Indeed, just like in (\ref{eq:identityideals}), non-binary models
are intersections of binary models:
\begin{equation}
 \mathcal{L}_{T,\geq}^{-1} \quad = \quad \bigcap_{\tilde U \in [\tilde T]}
\mathcal{L}_{U,\geq}^{-1} 
\end{equation}
Moreover, the inequalities for $T$ in (\ref{eq:constraints})
are those for binary $U$, as $\tilde U$ runs over $[\tilde T]$. Hence we can  assume that $T$ is binary. Suppose that $K\in \mathcal L_{T,\geq}^{-1}$.
By \cite[Theorem 2.2]{varga1993symmetric}, we know that
$K$ is positive definite and $p_{ij}\geq 0$ for all $0 \leq i<j \leq n$.
Theorem~\ref{thm:main} shows that $K$ satisfies the equalities in (\ref{eq:constraints}).
In  Lemma \ref{lem:key} we saw that the inequalities in (\ref{eq:constraints}) 
hold for $K$. Hence all constraints in (\ref{eq:constraints}) are satisfied for~$K$.
 
For the converse, let $K \in \mathbb{S}^n_+$ satisfy (\ref{eq:constraints}). By Theorem~\ref{thm:main}, the equations in (\ref{eq:constraints}) imply that $K\in \mathcal L_T^{-1}$. Since $K$ is invertible we can define $\Sigma=K^{-1}$ and $\Sigma=\Sigma_\theta$ for some real vector $\theta$. To complete the proof, we must show that $\theta$ is  nonnegative.

For any subset $A\subset \{1,\ldots,n\}$ denote by $T_A$ the tree whose vertices $V_A$ are 
 ${\rm lca}(i,j)$ for $i,j\in A$. There is a directed edge
  $u\to v$ in $T_A$ if there is a directed path from $u$ to $v$ in $T$ containing
   no other vertices of $T_A$. As before, we attach an 
   extra vertex $0$ to the root. Moreover, if $T$ has
   edge weights $\theta_v$ for $v\in V$ then $T_A$ has edge weights
\begin{equation}\label{eq:tildeth}
\tilde \theta_v\;=\;\sum_{u<w\leq v}\theta_w\qquad\mbox{ for }u\to v \,\,\mbox{ in }T_A.	
\end{equation}

For example, if $T$ is the tree in Figure~\ref{fig:running0}  and $A=\{1,2,3\}$ 
	then $T_A$ has vertices $V_A=\{1,2,3,5,7\}$ and edges 
	$0 {\to} 7$, $7 {\to} 3$, $7{\to} 5$, $5{\to} 1$, $5{\to} 2$. The weights of $T_A$~are
	$$
	\tilde \theta_1=\theta_1,\quad \tilde \theta_2=\theta_2,\quad \tilde \theta_3=\theta_3+\theta_6,\quad\tilde \theta_5=\theta_5,\quad \tilde \theta_7=\theta_7.
	$$

If $\Sigma $ lies in the subspace $ \mathcal{L}_{T} $ of $ \mathbb{S}^n$,
 with weights $\theta_w \in \mathbb R$, then its
principal submatrix $\Sigma_{A,A}$ lies in the subspace $\mathcal L_{T_A}$
of $ \mathbb{S}^{|A|}$. Indeed, the entries of  $\Sigma_{A,A}$~are
$$ \sigma_{ij}\;\, =\,\sum_{v\leq {\rm lca}(i,j)}\theta_v \quad \; =
\sum_{v\in V_A:\;v\leq {\rm lca}(i,j)}\!\!\!\tilde\theta_v.$$
	In other words, $\Sigma_{A,A}$ can be written as a matrix in $\mathcal L_{T_A}$ 
	with edge weights (\ref{eq:tildeth}).

As the main step in the proof, we will now show	
	that the constraints in (\ref{eq:constraints}) 
	behave nicely with respect to marginalization to the subtree
	induced on the subset $A$. Namely, we claim that
$\tilde K=(\Sigma_{A,A})^{-1}$ is a diagonally dominant M-matrix satisfying 
	\begin{equation} \label{eq:claim}
	\tilde p_{ik}\tilde p_{jl}\;=\;\tilde p_{il}\tilde p_{jk}\;\leq\; 
	\tilde p_{ij}\tilde p_{kl}
\end{equation}
	for all $i,j,k,l\in A\cup \{0\}$ such that the paths $\overline{ij}$, $\overline{kl}$ in $\tilde T_A$ have no edges in common.

The fact that $\tilde K = (\Sigma_{A,A})^{-1}$ is a diagonally dominant M-matrix follows directly from \cite[Corollary 2]{carlson1979schur}. To show the second part of the claim, we
shall assume $|A|=n-1$, say $A=\{1,\ldots,n-1\}$. The general case will then follow by
induction.

  If $K=\Sigma^{-1}$ then
    $\tilde K=K_{A,A}-\frac{1}{\kappa_{nn}}K_{A,n}K_{n,A}$,
    by taking the Schur complement.~Hence
	$$ \begin{matrix}
	\tilde p_{ij} & = & p_{ij}+\frac{1}{\kappa_{nn}}p_{in}p_{jn}
	&&& \hbox{for all $i,j \in A$} , \smallskip \\
	\tilde p_{0i} & = & p_{0i}+p_{in}-\frac{1}{\kappa_{nn}}p_{in}\sum_{j=1}^{n-1}p_{jn}
	& =& p_{0i}+\frac{1}{\kappa_{nn}}p_{0n}p_{in} & 
	\hbox{for $i \in A$}.
\end{matrix}	$$
For the quartet $ij|kl$ we conclude
	$$ \begin{matrix}
	\tilde p_{ik}\tilde p_{jl} & = & p_{ik}p_{jl}+\frac{1}{\kappa_{nn}}p_{ik}p_{jn}p_{ln}+\frac{1}{\kappa_{nn}}p_{in}p_{kn}p_{jl}+\frac{1}{\kappa_{nn}^2}p_{in}p_{jn}p_{kn}p_{ln},
	\smallskip \\
	\tilde p_{il}\tilde p_{jk} & = &p_{il}p_{jk}+\frac{1}{\kappa_{nn}}p_{il}p_{jn}p_{kn}+\frac{1}{\kappa_{nn}}p_{in}p_{ln}p_{jk}+\frac{1}{\kappa_{nn}^2}p_{in}p_{jn}p_{kn}p_{ln}.
\end{matrix}	$$
	By assumption, $p_{ik} p_{jl}=p_{il} p_{jk}$. We must show that the following
	expression is zero:
\begin{equation}\label{eq:auxquad}
		\tilde p_{ik}\tilde p_{jl}-\tilde p_{il}\tilde p_{jk}\;\;=\;\;\frac{1}{\kappa_{nn}}\left(p_{ik}p_{jn}p_{ln}+p_{in}p_{kn}p_{jl}-p_{il}p_{jn}p_{kn}-p_{in}p_{ln}p_{jk}\right).	
\end{equation}
Figure~\ref{fig:5cases} shows the five  cases of where $n$ can be located in $\tilde T$.
  First rewrite (\ref{eq:auxquad}) as
$$ \begin{matrix}
		\frac{1}{\kappa_{nn}}\left(p_{jn}(p_{ik}p_{ln}-p_{il}p_{kn})+p_{in}(p_{kn}p_{jl}-p_{ln}p_{jk})\right). \end{matrix}
$$
In the three cases in the top row of Figure~\ref{fig:5cases}, the paths $\overline{in}$ and $\overline{jn}$ do not intersect with the path $\overline{kl}$. This implies, by our assumption on $K$, that $p_{ik}p_{ln}-p_{il}p_{kn}=p_{kn}p_{jl}-p_{ln}p_{jk}=0$ and so (\ref{eq:auxquad}) is zero. For the remaining two cases we write (\ref{eq:auxquad})~as
$$ \begin{matrix}
		\frac{1}{\kappa_{nn}}\left(p_{ln}(p_{ik}p_{jn}-p_{in}p_{jk})+p_{kn}(p_{in}p_{jl}-p_{il}p_{jn})\right). \end{matrix}
$$
Since the path $\overline{ij}$ does not intersect the paths $\overline{kn}$ and $\overline{ln}$, 
we conclude the identities $\,p_{ik}p_{jn}-p_{in}p_{jk}=p_{in}p_{jl}-p_{il}p_{jn}=0$.
This again implies that (\ref{eq:auxquad}) is zero. 

	\begin{figure}[t]
		\includegraphics[scale=.74]{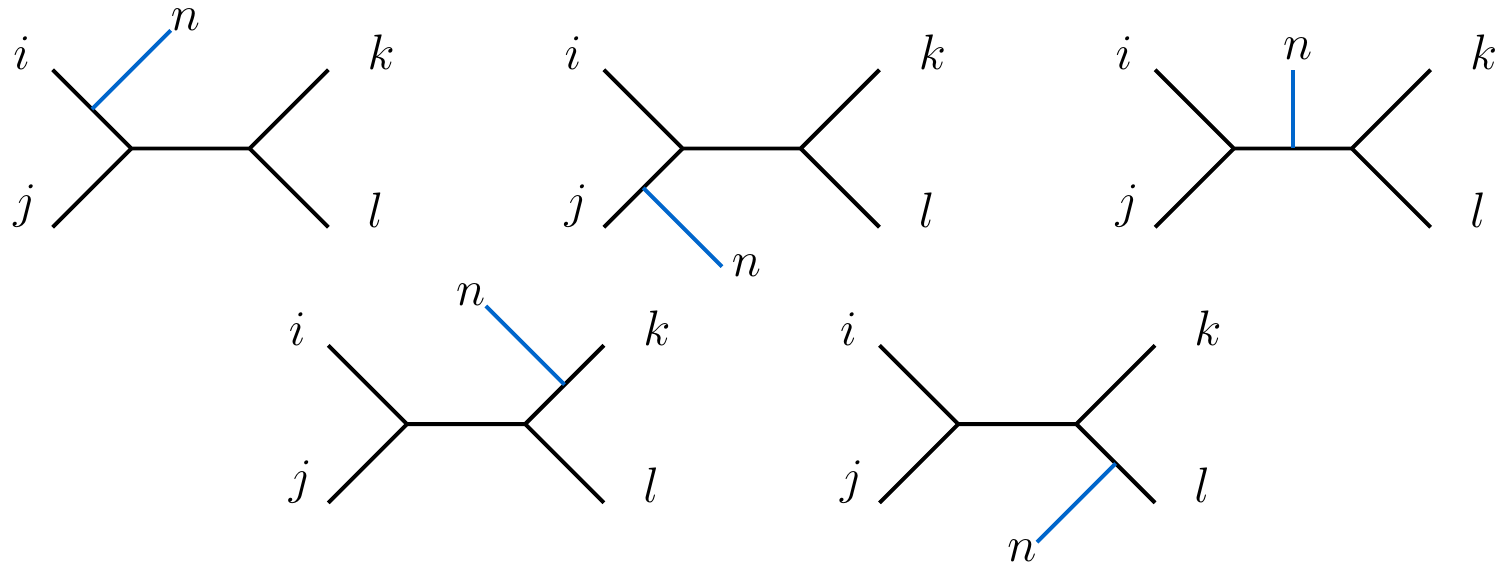} \vspace{-0.16in}
		\caption{The five cases for adding a leaf $n$ to the quartet $ij|kl$.}\label{fig:5cases}
	\end{figure}

It remains to show that $\tilde p_{il}\tilde p_{jk}\leq \tilde p_{ij}\tilde p_{kl}$. Similarly as above we obtain
	$$
	\tilde p_{il}\tilde p_{jk}- \tilde p_{ij}\tilde p_{kl}\;=\; (p_{il}p_{jk}- p_{ij} p_{kl})+\frac{1}{\kappa_{nn}}\left(p_{il}p_{jn}p_{kn}+p_{in}p_{ln}p_{jk}-p_{ij}p_{kn}p_{ln}-p_{in}p_{jn}p_{kl}\right).
	$$
	By assumption $p_{il}p_{jk}- p_{ij} p_{kl}\leq 0$.
	We will show that the second term,  denoted by $C$, is also nonpositive. 
Consider the five cases in Figure~\ref{fig:5cases} and write $C$ in two ways:
\begin{equation}\label{aux:C}
	\begin{matrix} C & = &	
	\frac{1}{\kappa_{nn}}\left(p_{kn}(p_{il}p_{jn}-p_{ij}p_{ln})+p_{in}(p_{ln}p_{jk}-p_{jn}p_{kl})\right)\\[.2cm]
& = & \frac{1}{\kappa_{nn}}\left(p_{jn}(p_{il}p_{kn}-p_{in}p_{kl})+p_{ln}(p_{in}p_{jk}-p_{ij}p_{kn})\right).	
	\end{matrix}	
\end{equation}	
	The following table shows the signs of the four relevant terms according to each case:
	\begin{center}
\begin{minipage}{0.5\textwidth}
	\begin{small} \begin{center}
\begin{tabular}{l|l|l|l|l|l}
  & 1 & 2 & 3 & 4 & 5 \\ \hline
$p_{il}p_{jn}-p_{ij}p_{ln}$ & $0$ & $+$ & $-$ & $-$ & $-$ \\
$p_{ln}p_{jk}-p_{jn}p_{kl}$ & $-$ & $-$ & $-$ & $0$ & $+$
\end{tabular}	
\end{center} \end{small}
\end{minipage}		\begin{minipage}{0.45\textwidth}
		\begin{small} \begin{center}
\begin{tabular}{l|l|l|l|l|l}
  & 1 & 2 & 3 & 4 & 5 \\ \hline
$p_{il}p_{kn}-p_{in}p_{kl}$ & $-$ & $-$ & $-$ & $+$ & $0$ \\
$p_{in}p_{jk}-p_{ij}p_{kn}$ & $+$ & $0$ & $-$ & $-$ & $-$
\end{tabular}	
\end{center} \end{small}
	\end{minipage}	\end{center}
Writing $C$ as in the first line of (\ref{aux:C}) implies nonpositivity in cases 1, 3, and 4. For the two remaining cases we use the second line. We conclude that $C\leq 0$ in all five cases. This completes the proof of the claim~ (\ref{eq:claim}).
 
We now finally show that $\theta_v\geq 0$ for every edge $u\to v$ in $T$.
Fix $A=\{i,j,k\}\subset \{1,\ldots,n\}$ such that $v={\rm lca}(i,j)$ and $k\in {\rm de}(u)\backslash {\rm de}(v)$ so that ${\rm lca}(i,k)={\rm lca}(j,k)=u$. Here we allow for $i=j$ if $v$ is a leaf and no $k$ if $u=0$. These two cases with $|A|=2$ will be considered separately; for now assume $|A|=3$. Consider the induced tree $T_A$. By construction, $u\to v$ is an edge of $T_A$. By the claim above, $\Sigma_{A,A}\in \mathcal L_{T_A}
$ is parameterized by $\tilde \theta$ with $\tilde \theta_v=\theta_v$.
Example~\ref{ex:mainsemismall} ensures that $\tilde \theta$ is a nonnegative vector; in particular $\tilde \theta_v=\theta_v\geq 0$. The case when $v$ is a leaf or when $u=0$ are similar, but here $|A|=2$, so we use the case $n=2$. This shows that, for any 
matrix $\Sigma$ satisfying the constraints (\ref{eq:constraints}), it follows
 that $\Sigma^{-1}\in \mathcal L_{T,\geq}^{-1}$. This completes the proof.
\end{proof}

Theorem~\ref{thm:mainsemi} offers a geometric understanding of
maximum likelihood estimation for Brownian motion tree models.
Given any sample covariance matrix $S$, the estimated concentration matrix $\hat K$ 
satisfies (\ref{eq:constraints}). If all inequalities are strict
for the estimates $\hat p_{ij}$ then we are in the situation of
 Section \ref{sec:mle}. Otherwise, we have $\hat p_{ij} = 0$ or
$\hat p_{il} \hat p_{jk} = \hat p_{ij} \hat p_{kl}$ for some choice
of indices in (\ref{eq:constraints}). This corresponds to
$\hat \Sigma = {\hat K}^{-1}$ lying on a proper face of the
simplicial cone $\mathcal{L}_{T,\geq}$. It is
interesting to record these faces.

\begin{example}[$n=4$] Fix the tree $T$ in Figure \ref{fig:running0}.
The following experiment was performed $1000$ times.
We fix the parameters $\,\theta_1=\cdots=\theta_7=1\,$
and the sample sizes $N=5$ and $N=20$.
We sample $N$ vectors from $\RR^4$ using the Gaussian
distribution $\Sigma_\theta$ and we
record the resulting sample covariance matrix $S$.
 In each case we compute
the MLE $\hat \Sigma$ using the standard function for constrained optimization in the statistical software {\tt R}.
For every iteration we check the KKT conditions to see whether the convergence criterion was met. In the affirmative case we identify the face of
the $7$-dimensional cone $\mathcal{L}_{T,\geq}$ that contains $\hat \Sigma$
in its relative interior. The following table shows the
empirical distribution of the codimension of the faces
that were found:
\begin{center}
\begin{tabular}{l|c|c|c|c|c}
codim & 0   & 1   & 2 & 3 & \textgreater{}3\\ \hline
$N=20$ & 816 & 183 & 1 & 0   &    0  \\ \hline    
$N=5$ & 487 &  374 & 119 & 20  & 0 \end{tabular}	
\end{center}
The numbers in the last column are zero because
 the faces of dimension less than four have empty 
 intersection with the cone of positive definite matrices.
 In the majority of the experiments, the MLE occurred in
 the interior of $\mathcal{L}_{T,\geq}$. In this case,
 the analysis in Example \ref{ref:galois5} applies:
  the MLE $\hat \Sigma$ has algebraic degree five
  over the data $S$. \hfill \qed
 \end{example}

Every face of the simplicial cone $\mathcal{L}_{T,\geq}$
has the form  $\mathcal{L}_{T',\geq}$,
where $T'$ is obtained from $T$ by contracting some edges. If
MLE $\hat \Sigma $ lies on that face, then
the algebraic complexity of the MLE is governed by the ML degree for $T'$.
This underscores the relevance of results like  Proposition \ref{prop:mldstar},
even if the tree $T$ of interest is not binary.

Theorem~\ref{thm:mainsemi} implies that the
 facial structure of the simplicial cone $\mathcal{L}_{T,\geq}$
translates into a stratification of the boundary of
 $\mathcal{L}_{T,\geq}^{-1}$. This enables a detailed
   geometric analysis of the MLE
 across all strata. We shall pursue this in a forthcoming paper.

\end{document}